\documentclass[12pt]{article}
\usepackage{amsmath,amssymb,amsthm}

\begin{document}

\def\si{\par\smallskip\noindent}
\def\bi{\par\bigskip\noindent}
\def\pr{\mathop{\rm Pr}}
\def\ex{E}
\def\W{W}
\def\de{\delta}
\def\ds{\displaystyle}
\def\eps{\varepsilon}
\def\la{\lambda}
\def\a{\alpha}
\def\be{\beta}
\def\de{\Delta}
\def\sig{\sigma}
\def\ga{\gamma}
\def\part{\partial}
\def\Cal{\mathcal}
\def\P{{\mathcal P}}
\def\G{{\mathcal G}}
\def\M{{\mathcal M}}
\def\z{{\bf z}}
\def\var{\text{Var\/}}
\def\Re{\mathop{\rm Re}}
\def\Im{\mathop{\rm Im}}
\def\Prob{\mathop{\rm Prob}}
\def\Var{\mathop{\rm Var}}
\newtheorem{Theorem}{Theorem}[section]
\newtheorem{Lemma}[Theorem]{Lemma}
\newtheorem{Proposition}[Theorem]{Proposition}
\newtheorem{Corollary}[Theorem]{Corollary}
\newtheorem{Claim}{Claim}[section]
\theoremstyle{definition}
\newtheorem{Definition}[Theorem]{Definition}
\newtheorem{Remark}[Theorem]{Remark}
\newtheorem{Remarks}[Theorem]{Remarks}
\newtheorem{Example}{Example}[section]
\def\be{\begin{equation}}\def\ee{\end{equation}}
\def\bbc{\mathbb{C}}
\def\bbl{\mathbb{L}}
\def\bbn{\mathbb{N}}
\def\bbr{\mathbb{R}}
\def\bbz{\mathbb{Z}}
\def\dmax{d_{\rm max}}
\newcommand{\csect}[1]{Section~\ref{#1}}
\newcommand{\capp}[1]{Appendix~\ref{#1}}
\newcommand{\cthm}[1]{Theorem~\ref{#1}}
\newcommand{\ccor}[1]{Corollary~\ref{#1}}
\newcommand{\clem}[1]{Lemma~\ref{#1}}
\newcommand{\cprop}[1]{Proposition~\ref{#1}}
\newcommand{\crem}[1]{Remark~\ref{#1}}
\newcommand{\cex}[1]{Example~\ref{#1}}
\newenvironment{proofof}[1]{\medskip\noindent
   \textit{Proof of #1.} }{\hfill \qed\par\medskip}
\newenvironment{example}{\medskip\noindent
   \textit{Example.}}{\par\medskip}
\renewcommand{\labelenumi}{\bf{A\arabic{enumi}}}
\newcounter{hold}
\newcommand{\pause}{\setcounter{hold}{\value{enumi}}\end{enumerate}}
\newcommand{\resume}{\begin{enumerate}\setcounter{enumi}{\value{hold}}}
\numberwithin{equation}{section} 

\title{\null\vskip-60pt 
  Central limit theorems, Lee-Yang zeros, and graph-counting polynomials}
\date{\today}
\author{{J. L. Lebowitz${}^{1,2}$, B. Pittel${}^{3}$, D. Ruelle${}^{1,4}$, 
     and E. R. Speer${}^1$ }\\ \\
{\small $^1$ Department of Mathematics, Rutgers University,}\\
{\small Piscataway NJ 08854-8019 USA}\\
{\small $^2$ Department of Physics, Rutgers University,}\\
{\small Piscataway NJ 08854-8019 USA}\\
{\small $^3$ Department of Mathematics, The Ohio State University,}\\
{\small 231 W. $18^{\rm th}$ Avenue, Columbus, OH 43210 USA}\\
{\small $^4$ IHES, 91440 Bures sur Yvette, France}}

\maketitle



\begin{abstract}We consider the asymptotic normalcy of families of random
variables $X$ which count the number of occupied sites in some large set.
We write $\Prob(X=m)=p_mz_0^m/P(z_0)$, where $P(z)$ is the generating
function $P(z)=\sum_{j=0}^{N}p_jz^j$ and $z_0>0$.  We give sufficient
criteria, involving the location of the zeros of $P(z)$, for these families
to satisfy a central limit theorem (CLT) and even a local CLT (LCLT); the
theorems hold in the sense of estimates valid for large $N$ (we assume that
$\Var(X)$ is large when $N$ is).  For example, if all the zeros lie in the
closed left half plane then $X$ is asymptotically normal, and when the
zeros satisfy some additional conditions then $X$ satisfies an LCLT.  We
apply these results to cases in which $X$ counts the number of edges in the
(random) set of ``occupied'' edges in a graph, with constraints on the
number of occupied edges attached to a given vertex.  Our results also
apply to systems of interacting particles, with $X$ counting the number of
particles in a box $\Lambda$ whose size $|\Lambda|$ approaches infinity;
$P(z)$ is then the grand canonical partition function and its zeros are the
Lee-Yang zeros.  \end{abstract}

\section{Introduction\label{intro}}

In this note we investigate the asymptotic normalcy of the number $X$ of
elements in a random set $M$ when the expected size of $M$ is very large.  We
shall be concerned in particular with the case in which $M$ is a random
set  of edges, called {\it occupied edges}, in some large graph
$G$, under certain rules which constrain the admissible configurations of
occupied edges.  Our analysis is however not restricted to such examples; in
particular, it includes many cases of interest in statistical mechanics,
for which $X$ is the number of occupied sites in some region
$\Lambda\subset \bbz^d$ (or the number of particles in
$\Lambda\subset \bbr^d$).

The probability that $X=m$ is written as
 \be\label{Xdef}
  \Prob\{X=m\}:=\frac{p_mz_0^m}{P(z_0)}\;,
 \ee
 where
 \be\label{Pdef}
  P(z):=\sum_{m=0}^N p_mz^m
 \ee
 is a polynomial of degree $N$ and $z_0$ is a strictly positive parameter;
we will often take $z_0=1$.  The coefficient $p_m$ will be, in the graph
counting case, the number of admissible configurations of occupied edges of
size $m$.  By convention we take $p_m=0$ if $m>N$ or $m<0$.  In some
cases we will consider $P$ as the fundamental object of study and will
then write $X_P$ and $N_P$ for $X$ and $N$.

A simple example is that in which a configuration is admissible if the
number of occupied edges attached to each vertex $v$, $d_M(v)$, is zero or
one.  In this case the polynomial $P(z)$ coincides
with one of several definitions of the {\it matching polynomial} of the
graph, properties of which have been studied extensively in the graph
theory literature.  In particular, a local central limit theorem (see
below) for $X$ has been proved in the case $z_0=1$ \cite{GLCLT}.  Our
primary examples in this paper will be {\it graph-counting polynomials},
which arise when the restriction $d_M(v)\in\{0,1\}$ discussed above is
generalized to $d_M(v)\in C(v)$ for some set $C(v)$; we will obtain a local
central limit theorem for $X$ when  $C(v)=\{0,1,2\}$ for all $v$.

The above examples are also natural objects of study in equilibrium
statistical mechanics; there one refers to the case with $d_M(v)\in\{0,1\}$
as a system of {\it monomers} and {\it dimers}, and to that with
$d_M(v)\in\{0,1,2\}$ as a system of monomers and {\it unbranched polymers}.
In this setting one thinks of the edges belonging to $M$ as occupied by
particles, and the parameter $z_0$ is then the {\it fugacity} of these
particles. The restriction $d_M(v)\in C(v)$ with $C(v)=\{0,1,\ldots,c_v\}$
corresponds to {\it hard core} interactions between the particles, and is a
special or limiting case of a more general model for which a configuration
$M$ is assigned a Gibbs weight $w_M:=e^{-\beta U(M)}$, with $U(M)$ the {\it
interaction energy} of $M$ and $\beta$ the inverse of the temperature, and
$p_m:=\sum_{\{M\mid|M|=m\}}w_M$.  $p_m$ is then called the {\it canonical
partition function} for $m$ particles and $P(z)$ the {\it grand canonical
partition function} of the system.

In this statistical mechanics setting the graph $G$ is usually a subset of
a regular lattice.  For example, the vertices may be the sites of the
lattice $\bbz^d$ which belong to some cubical box
$B=\{1,\ldots,L\}^d\subset\bbz^d$, with edges, usually called bonds,
joining nearest-neighbor sites; one also considers such a box with periodic
boundary conditions, in which an additional bond joins any pair of sites
whose coordinate vectors differ in only one component, in which the values
for the two sites are are $1$ and $L$.  Such a box contains $|B|$ vertices
and $\sim d|B|$ edges.  The particles are most often thought of as
occupying the sites of the lattice, that is, the vertices of the graph, but
for our examples they occupy the bonds, as noted above.  For the
monomer-dimer problem on such a box $B$ one would have $N\sim|B|/2$.
Considering potentials U for the periodic box which are translation
invariant and sufficiently regular we are then in the usual situation for
equilibrium statistical mechanics, see e.g.  \cite{DR2, Georgii}.

 In the statistical mechanics setting there are many cases in which one can
prove that $E[X]\sim c_1N$ and $\Var(X)\sim c_2N$ for some $c_1,c_2>0$ and
that $X$ satisfies a central limit theorem (CLT), that is, that
 \be\label{CLT}
\Prob\bigl\{\,X\le E[X]+x\sqrt{\Var(X)}\,\bigr\}\sim G(x)
 \ee
  when $N\to\infty$, where $G(x)$ is the cumulative distribution function
of the standard normal random variable. A discussion of different proofs is
given in \cite[p.~469]{Georgii}; most of these make use of the approximate
independence of distant regions of $\bbz^d$ to write $X$ as a sum of many
approximately independent variables, and do not extend directly to general
graphs without any spatial structure. See also \cite{Canfieldsurvey} for a broad review of proof
methods in the context of combinatorial enumeration. Here, inspired by a proof due to
Iagolnitzer and Souillard \cite{IS} in a statistical mechanics context, we
prove a CLT that requires only that for large $N$ there be no zeros of
$P(z)$ in some disc of uniform size around $z_0$, and that $\Var(X)$ grow
faster than $N^{2/3}$ as $N\to\infty$.  We describe the method in
\csect{statmech} and in \csect{CLTSM} verify the variance condition, and
thus obtain a CLT, for the random variables associated with a class of
graph-counting polynomials and for the particle number in some statistical
mechanical systems.  We note here and will show later that when the zeros
of $P(z)$ lie in the left half plane it is sufficient for the CLT that
$\Var(X)\to\infty$ as $N\to\infty$.

Once one has a CLT for $X$, in the usual sense \eqref{CLT} of convergence of
distributions, one would like also a {\it local} CLT (LCLT), that is, one
would like to show that for large $N$,
 \be\label{LCLT}
 \Prob\{X=m\}\sim \frac1{\sqrt{2\pi\Var(X)}}e^{-(m-E[X])^2/2\Var(X)}.
 \ee
 If \eqref{LCLT} holds for $m$ belonging to some set $S$ of integers then
one speaks of an LCLT {\it on $S$}, but in the cases we will consider we
will prove an LCLT on all of $\bbz$.  In the statistical mechanics setting
such a result was established for certain systems in \cite{DT}; see also
\cite{Georgii}.  An LCLT for dimers on general graphs was given by Godsil
\cite{GLCLT}, with a very different proof.  Earlier Heilmann and Lieb
\cite{HL} proved that all the zeros of the attendant matching polynomial
$P(z)$, whose coefficients $p_m$ enumerate incomplete matchings
(monomer-dimer configurations) by the number $m$ of edges (dimers), lie on
the negative real axis. Harper \cite{Harper} was the first to
recognize---in a particular case of Stirling numbers---that such a property
of a generating function $P(z)$ meant that the distribution of the
attendant random variable is one of a sum of independent, $(0,1)$-valued,
random variables; it instantly opened the door for his proof of asymptotic
normality of those numbers. Godsil used Heilmann-Lieb's result and Harper's
method to prove a CLT for $\{p_m\}$, under a constraint on the ground graph
guaranteeing that the variance tends to infinity.  Significantly, since
Heilmann-Lieb's result and Menon's theorem \cite{Menon} implied
log-concavity of $\{p_m\}$, Godsil was able to prove the stronger LCLT by
using the quantified version of Bender's LCLT for log-concave distributions
\cite{Bender} due to Canfield \cite{Canfield}. We refer the reader to Kahn
\cite{Kahn} for several necessary and sufficient conditions under which the
variance of the random matching size tends to infinity, and to Pitman
\cite{Pitman} for a broad range survey of the probabilistic bounds when the
generating function has real roots only.

Years later Ruelle \cite{Ruelle3a} found that the polynomial $P(z)$ whose
coefficients enumerate the {\it unbranched\/} subgraphs ($2$-matchings)
of a general graph $G$ has roots in the left half of the $z$-plane, but
not necessarily on the negative real line. Our key observation is that
here again the related random variable $X$ is, in distribution, a sum of
independent random variables, this time each having a $3$-element range
$\{0,1,2\}$. Since the range remains bounded, a CLT for unbranched
polymers follows whenever $\text{Var}X$ goes to infinity with the
degree of $P$. However, only when the roots are within a certain wedge
enclosing the negative real axis can we prove log-concavity of the
distribution of $X$. Still we are able to prove an LCLT, with an explicit
error term, under certain mild conditions on $G$.

We now summarize briefly some consequences of our results (not
necessarily the optimal ones).  Assuming that the mean $E[X]$ and
variance $\Var(X)$ go to infinity as $N\to\infty$, then:

 \smallskip\noindent
 1. The random variable $X$ satisfies a CLT for all $z_0>0$ if all roots
$\zeta$ of $P$ satisfy $\Re\zeta\le0$.

 \smallskip\noindent
 2. The random variable $X$ satisfies an LCLT for all $z_0>0$ (a) if all
roots $\zeta$ are in a wedge of angle $2\pi/3$ centered on the negative
real axis, and (b) if $\Re\zeta\le-\delta$, $\delta>0$, and $\Var(X)$ grows
faster than $N^{2/3}$.

 \smallskip\noindent
 3. The random variable $X$ satisfies a CLT if there are no zeros of $P$ in
a disc of radius $\delta>0$ around $z_0$ and $\var(X)$ grows faster than
$N^{2/3}$ (see \cite{IS}).

 \smallskip\noindent
 4. Finally, we show that certain of the above conditions are satisfied by
many graph-counting polynomials and statistical mechanical systems---for
example, unbranched polymers---and hence obtain a CLT or LCLT in these
cases.  The result mentioned in 2(a) above has also been used \cite{FL} to
establish an LCLT for determinantal point processes.

 \smallskip
 
The outline of the rest of the paper is as follows. In \csect{statmech} we
apply the method of \cite{IS} to derive a CLT for the random variable $X$
from rather weak hypotheses on the location of the zeros of $P(z)$, and in
\csect{lhplane} we obtain an LCLT under the stronger hypothesis that the
zeros lie in the left half plane.  In \csect{gcpolys} we describe more
precisely the class of graph-counting polynomials and what can be said
about the location of their zeros.  In \csect{applications} we obtain
central limit theorems and, in some cases, local central limit theorems for
graph-counting polynomials from the results of \csect{lhplane}, and in
\csect{CLTSM} obtain, from the results of \csect{statmech}, central limit
theorems for further graph-counting examples and for some statistical
mechanical systems. Throughout our discussions we will, rather than
considering sequences of polynomials, say that a family $\P$ of
polynomials, of unbounded degrees, satisfies a CLT or an LCLT when one can
give estimates for the errors in the approximations \eqref{CLT} and
\eqref{LCLT}, respectively, which are valid for all polynomials in $\P$ and
which vanish as the degree $N$ of the polynomial goes to infinity.

\section{A central limit theorem}\label{statmech}

In this section we first consider a fixed polynomial
$P(z)=\sum_{m=0}^{N}p_mz^m$, as in \eqref{Pdef}, and assume throughout that
$p_m\ge0$ and that $p_N>0$, i.e., that $P$ is in fact of degree $N$.  We
fix also a number $z_0>0$ (a fugacity, in the language of statistical
mechanics) and let $X$ to be a random variable with probability
distribution given by \eqref{Xdef}.  We will let $\zeta_j$, $j=1,\ldots,N$,
denote the roots of $P$.

Our first result is an estimate corresponding to an (integrated) central
limit theorem.  To state it we define, for $x\in\bbr$,
 \begin{eqnarray}
F(x)&:=&\frac1{P(z_0)}\sum\limits_{m\le E[X]+ x\sqrt{\var(X)}}p_mz_0^m   =
\Prob\left\{\frac{X-E[X]}{\sqrt{\Var(X)}}\le x\right\},\\
G(x)&:=&(2\pi)^{-1/2}\int_{-\infty}^xe^{-u^2/2}\,du.
 \end{eqnarray}

\begin{Theorem}\label{stronger} Suppose that there exists a $\delta>0$ such
that $z_0\ge\delta$ and $|z_0-\zeta_j|\ge\delta$ for all $j$,
$j=1,\ldots,N$.  Then there exist constants $N_0,B_1,B_2>0$, depending only on
$\delta$ and $z_0$, such that for $N\ge N_0$, 
 \be\label{likeesseen}
\sup_{x\in \bbr}|F(x)-G(x)|
    \le\frac{B_1N}{\Var(X)^{3/2}}+ \frac{B_2N^{1/3}}{\Var(X)^{1/2}}.
 \ee
\end{Theorem}

\begin{Remark}\label{standard}  We record here some standard results,
adopting the notation of \cthm{stronger}.  For $z$ in the disk
$D:=\{z\in\bbc\mid|z-z_0|<\delta\}$ we will define $\log P(z)$ by
 \be\label{logPdef}
\log P(z):=\log p_N+\sum_{j=1}^N \log\bigl(z-\zeta_j\bigr),
 \ee
 with $\log p_N$ real and 
 \be\label{arg1}
\log(z-\zeta_j):=\log(z_0-\zeta_j)+\log\frac{z-\zeta_j}{z_0-\zeta_j}, 
 \ee
 where
 \be\label{arg2}
   \Im\log(z_0-\zeta_j)\in(-\pi,\pi)\quad\text{and}\quad
   \Im\log\frac{z-\zeta_j}{z_0-\zeta_j}\in(\-\pi/2,\pi/2).
 \ee
 In \eqref{arg2} the first specification is possible since $\zeta_j$
cannot be a positive real number and the second since
$\bigl|(z-\zeta_j)/(z_0-\zeta_j)-1\bigr|<1$ for $z\in D$; in particular,
$\log (z-\zeta_j)/(z_0-\zeta_j)$ is analytic for $z\in D$.  Moreover, $\log P(z)$ is
real for real $z$, because non-real roots occur in complex conjugate
pairs, and furthermore
\begin{equation}\label{logdef}
\log P(z) -\log P(z_0)=\sum_{j=1}^N \log\frac{z-\zeta_j}{z_0-\zeta_j},
   \quad z\in D.
\end{equation}
   Then for all $z$ in $D$,
 \begin{eqnarray}
z\frac{d}{dz}\log P(z)&=&\frac{\sum_m mp_mz^m}{P(z)},\nonumber\\
\left(z\frac{d}{dz}\right)^2\log P(z)
  &=&\frac{\sum_mm^2p_mz^m}{P(z)}-\left(\frac{\sum_m mp_mz^m}{P(z)}\right)^2,
 \end{eqnarray}
 and so 
 \be\label{stdz}
  z\frac{d}{dz}\log P(z)\Bigl|_{z=z_0}=E[X],\quad
\left(z\frac{d}{dz}\right)^2\log P(z)\Bigl|_{z=z_0}=\Var(X).
 \ee
 From \eqref{stdz} we also have
 \be\label{stdu}
  \frac{d}{du}\log P(e^uz_0)\Bigl|_{u=0}=E[X],\quad
\frac{d^2}{du^2}\log P(e^uz_0)\Bigl|_{u=0}=\Var(X).
 \ee
\end{Remark}

 To state the next lemma we observe that there exists an $\epsilon>0$,
depending only on $\delta$ and $z_0$, such that if $|u|\le\epsilon$ then
$|e^u z_0-z_0|\le\min\{\delta/2,|z_0|\}$, so that for  $|u|\le\epsilon$
we may define, as in \crem{standard}, 
\begin{eqnarray}
f(u):=\log E[e^{uX}]&=&\log P(e^uz_0) - \log P(z_0)
   \nonumber\\
  &=&\sum_{j=1}^N \log\frac{e^u z_0-\zeta_j}{z_0-\zeta_j}.
  \label{fdef}
 \end{eqnarray}

 \begin{Lemma}\label{lem} Let $\delta$ be as in \cthm{stronger} and let
$\epsilon=\epsilon(z_0,\delta)$ be as above. Then for $K=2\log2/\epsilon^3$,
 \be
 f(u)=uE[X]+\frac{u^2}{2}\Var(X) +u^3R(u), 
    \quad \text{ with }\quad |R(u)|\le NK.
 \ee
  \end{Lemma}

\begin{proof} Suppose that $|u|\le\epsilon/2$.  Then we have, by Cauchy's
integral formula and \eqref{stdu},
 \begin{eqnarray}\label{fexp}
f(u)&=&f(0)+u f'(0)+\frac{u^2}2 f''(0)+u^3R(u)\nonumber\\
  &=&uE[X]+\frac{u^2}{2}\Var(X) +u^3R(u),
 \end{eqnarray}
 where
 \be\label{Rdef}
R(u):=\frac1{2\pi i}\oint_{|v|=\epsilon}\frac{f(v)}{v^3(v-u)} \,dv.
 \ee
  Then from \eqref{fdef},
 \begin{eqnarray}
|R(u)|
&\le& \sum_{j=1}^N \left|\frac1{2\pi i}\oint_{|v|=\epsilon} 
   \log\left(\frac{e^v z_0-\zeta_j}{z_0-\zeta_j}\right)
   \,\frac{dv}{v^3(v-u)}\right|\nonumber\\
 &\le&\frac{2}{\epsilon^3}\sum_{j=1}^N\sup_{|v|=\epsilon}
 \left|\log \frac{e^vz_0-\zeta_j}{z_0-\zeta_j}\right|
 < \frac{2}{\epsilon^3}\,N\log2.\label{Rbd}
 \end{eqnarray}
 Here we have used
$|(e^vz_0-\zeta_j)/(z_0-\zeta_j)|<(\delta/2)/\delta=1/2$ for
$|v|=\epsilon$ and $|\log(1-t)|\le-\log(1-|t|)$ for $|t|<1$; the latter
is easily verified for example from the expansion
$\log(1-t)=-\sum_{k\ge1}t^k/k$.  \end{proof}

\begin{proofof}{Theorem \ref{stronger}} The proof follows closely the proof
of the Esseen-Berry Theorem given in Feller \cite[Section XVI.5]{F} and in
particular is based on the ``smoothing inequality'' \cite[Section XVI.4,
Lemma 2]{F}.  If we specialize to the particular application we need then
the latter implies that for any $T>0$,
 \be\label{smooth}
\sup_{x\in \bbr}|F(x)-G(x)|\le 
  \frac{1}{\pi}\int_{-T}^T \left|\frac{\psi(t)-e^{-t^2/2}}{t}\right|\,dt
    + \frac{24}{\pi\sqrt{2\pi}T},
 \ee
 where $\psi(t)=E[e^{itY}]$ is the characteristic function of
$Y=(X-E[X])/\sigma$, with $\sigma=\sqrt{\var(X)}$.  We will apply this
inequality with $T=\sigma/N^{1/3}$.  For $|t|\le T$, then,
$|t/\sigma|\le N^{-1/3}$, so that for $N\ge N_0:=8/\epsilon^3$ we have
$t/\sigma\le\epsilon/2$ and, from Lemma~\ref{lem},
 \be
\psi(t)=e^{-itE[X]/\sigma}e^{f(it/\sigma)}
     =e^{-t^2/2-it^3R(it/\sigma)/\sigma^3},
 \ee
 with $|R(it/\sigma)|\le NK$ and hence $|it^3R(it/\sigma)/\sigma^3|\le K$.
Now let $K_*=\max_{|u|\le K}|(e^{iu}-1)/u|$, so that 
 \be
|e^{-it^3R(it/\sigma)/\sigma^3}-1|\le |t/\sigma|^3NKK_*  \quad 
  \text{for $N\ge8/\epsilon^3$ and $t\le T$}.
 \ee
  Then 
 \begin{eqnarray}
\int_{-T}^T \left|\frac{\psi(t)-e^{-t^2/2}}{t}\right|\,dt
  &\le&\frac{NKK_*}{\sigma^3}\int_{-T}^Tt^2e^{-t^2/2}\,dt\nonumber\\
  &\le&\frac{NKK_*}{\sigma^3}\int_{-\infty}^\infty t^2e^{-t^2/2}\,dt
  = \frac{NKK_*\sqrt{2\pi}}{\sigma^3}.
 \end{eqnarray}
 Inserting this estimate into \eqref{smooth} we obtain \eqref{likeesseen}
 with
 \be
  B_1:= \sqrt{\frac{2}{\pi}}\,KK_*, \qquad 
  B_2:=\frac{24}{\pi\sqrt{2\pi}}.
 \ee
 \end{proofof}

In \csect{CLTSM} we will apply \cthm{stronger} to obtain central limit
theorems for families of graph-counting polynomials and for families of
polynomials arising from statistical mechanics.  To do so we must establish
that, for $P$ in the family under consideration, $\Var(X_P)$ grows faster than
$N_P^{2/3}$.  Our tool for this will be a result due to Ginibre \cite{JG},
which we recall as \cthm{ginibre} below; our next result, which is similar
to \cthm{stronger}, will be needed in the application of Ginibre's result
to graph-counting polynomials.

\begin{Proposition}\label{mean} Suppose that $p_0$ and $ p_1$ are nonzero
and that $c_1$ and $\delta_1$ are positive constants such that
(i)~$p_1\ge c_1p_0N$ and (ii)~$|\zeta_j|\ge\delta_1$, $j=1,\ldots,N$.
Then there exists a constant $M>0$, depending only on $c_1$, $\delta_1$,
and $z_0$, such that $E[X]\ge MN$.  \end{Proposition}

\begin{proof} For $z$ real and nonnegative, $\log P(z)$ is well defined by
the requirement that it be real;  further,
 \be
E[X]=z\frac{d}{dz}\log P(z)\big|_{z=z_0}
 \ee
  and
 \be
 z_0\frac{d}{dz_0}E[X]=\Var(X)>0,
 \ee
  so that $E[X]$ is an increasing function of $z_0$. Thus it suffices to
verify the conclusion for sufficiently small $z_0$. Now we allow $z$ to be
complex, and for $|z|<\delta_1$ define as in \eqref{logdef}
 \be
 g(z):=\log P(z)-\log P(0)=
  \sum_{j=1}^N\log\frac{\zeta_j-z}{\zeta_j}, 
 \ee
  where again $\Im\log((\zeta_j-z)/\zeta_j)\in(-\pi/2,\pi/2)$.  Now for
$|z|<\delta_1/4$ we have
  \begin{eqnarray}
  zg'(z)&=&z\frac{d}{dz}\left(g(0)+zg'(0)
  +\frac{z^2}{2\pi i}\oint_{|y|=\delta_1/2}\frac{g(y)}{y^2(y-z)} \,dy\right)
  \nonumber\\
 &=&z\frac{p_1}{p_0}+z^2R_1(z), 
 \end{eqnarray}
 with
 \be
  R_1(z):=\frac1{2\pi i}
    \oint_{|y|=\delta_1/2}\frac{(2y-z)\,g(y)}{y^2(y-z)^2} \,dy.
 \ee
 Since for $|y|=\delta_1/2$ and $|z|\le\delta_1/4$ we have
$1/|y|^2=4/\delta_1^2$, $1/|y-z|\le4/\delta_1$, $|2y-z|<5\delta_1/4$, and
$|g(y)|\le\log2$ (see \eqref{Rbd}), we find that
 \be
|R_1(z)|
\le \frac{40}{\delta_1^2}\,N\log2.\label{Rbd23}
 \ee
 Let $z_*=\min\{\delta_1/4,c_1\delta_1^2/(80\log2)\}$; then for
$0<z_0\le z_*$ ,
 \be\label{ps3}
 E[X]=zg'(z)\Big|_{z=z_0}\ge z_0\frac{p_1}{p_0} 
   - \frac{40z_0^2}{\delta_1^2}\,N\log2 
   \ge \frac{z_0c_1N}{2}.
 \ee
 Thus $E[X]\ge MN$ holds with $M=z_0c_1/2$ for $z_0\le z_*$ and with
 $M=z_*c_1/2$ otherwise.
\end{proof}

\begin{Remark} \cthm{stronger} strengthens and gives a complete proof of
the result in \cite{IS} that $F(x)\to G(x)$ as $N\to\infty$. \cite{IS}
considered specifically the Ising model, for which it is known that
$\Var(X)\ge cE(X)\ge kN$, $c,k>0$; see \csect{CLTSM}.  We also note here
that Dobrushin and Shlosman \cite{DS} proved a local ``large and moderate
deviation'' result for $X$ which implies a LCLT under a further {\it
locality} condition, which rules out situations in which all the zeros
are close to the imaginary axis.  The locality condition is in turn
implied by a certain bound on the characteristic function $E[e^{itX}]$,
which they showed to hold for the Ising model at zero magnetic field and
high temperature.  The bound in question is somewhat stronger than the
bound \eqref{maxphi1} which we obtain from the condition that the roots
all lie in the negative half of the complex plane.  \end{Remark}

\section{Polynomials with zeros in the left half\\ plane}\label{lhplane}

In this section we again consider a polynomial $P(z)$ as in \eqref{Pdef},
and continue to assume that $P$ is of degree $N$ and that all the
coefficients $p_m$ are nonnegative.  Moreover, we assume that all roots
of $P$ lie in the closed left-half plane, and no root is zero, i. e.
$p_0>0$. For convenience we now write these roots as $-\eta_j$, so that
 \be\label{clhp}
  \text{Re}(\eta_j)\ge 0,\,(j=1,\dots,N), \quad\text{and}\quad P(z)=p_N\prod_{j=1}^N(z+\eta_j).
 \ee
 We will take the fugacity $z_0$ to be 1, but our results extend easily
to any $z_0>0$.

\subsection{A  central limit theorem\label{3CLT}}

Under the assumption \eqref{clhp} the derivation of a CLT given in
\csect{statmech} can be simplified; moreover, the result is strengthened
since we require only that $\Var(X_P)\to\infty$ as $N_P\to\infty$, in
contrast to the power growth condition needed to apply \cthm{stronger}.
The key idea is to write $X_P$ as a sum of independent random variables;
the central limit theorem then follows, for example from the Berry-Esseen
theorem.  In the case in which all the $\eta_j$ are nonnegative the method
goes back to Harper \cite{Harper}.  
\cite{Canfield}, 

To decompose $X_P$ as such a sum, we partition $\{1,\ldots,N\}$ as
$J_1\cup J_2\cup J'_2$, where $j\in J_1$ iff $\eta_j$ is real and
$j\in J_2$ (respectively $j\in J'_2$) iff $\Im (\eta_j)>0$ (respectively
$\Im(\eta_j)<0$); the corresponding factorization of $P(z)$ is
 \be\label{fact}
P(z)=p_N\prod_{j\in J_1}(z+\eta_j)\prod_{j\in J_2} (z^2+2\Re(\eta_j)
z+|\eta_j|^2).
 \ee
  We then introduce independent random variables $X_j$, $j\in J_1\cup J_2$,
where if $j\in J_1$ (respectively $j\in J_2$) then $X_j$ takes values $0$
and $1$ (respectively $0$, $1$, and $2$). With $P_j(z)=z+\eta_j$ for
$j\in J_1$ and $P_j(z)=z^2+2z\Re(\eta_j)+|\eta_j|^2$ for $j\in J_2$, the
individual distribution of these random variables is
 \begin{gather*}
\pr\{X_j=0\}=\frac{1}{P_j(1)},\quad
 \pr\{X_j=1\}=\frac{\eta_j}{P_j(1)},\qquad (j\in J_1);\\
 \left.\begin{gathered}
 \pr\{X_j=0\}=\frac{|\eta_j|^2}{P_j(1)},\quad
  \pr\{X_j=1\}=\frac{2\Re(\eta_j)}{P_j(1)},\\
    \pr\{X_j=2\}=\frac{1}{P_j(1)},\end{gathered}\quad\right\}\quad (j\in J_2).
\end{gather*}
 Then $E[z^{X_j}]=P_j(z)/P_j(1)$ and so
 \be
E[z^{\sum_{j\in J_1\cup J_2}X_j}]
   =\prod_{j\in J_1\cup J_2}\frac{P_j(z)}{P_j(1)}
 =\frac{P(z)}{P(1)}=E[z^{X_P}]
 \ee
 for all $z$. Thus $X_P$ and $\sum_{j\in J_1\cup J_2}X_j$ have the same
distribution, and we may identify these two random variables.

\begin{Theorem}\label{clt2} Let $\P$ be a family of polynomials as
in \eqref{Pdef}, of unbounded degrees, all of which satisfy \eqref{clhp}.
Then for each $P\in\P$,
 \be\label{esseen3}
\sup_{x\in \bbr}|F_P(x)-G(x)|\le\frac{12}{\sqrt{\Var(X_P)}}.
 \ee
 Consequently, if $\var(X_P)\to\infty$ as $N_P\to\infty$ in $\P$ then $\P$
satisfies a CLT in the sense described in \csect{intro}.   \end{Theorem}

\begin{proof} From \cite[Section XVI.5, Theorem 2]{F} and $|X_j|\le2$ we
have immediately that the left hand side of \eqref{esseen3} is bounded by
 \be
\frac{6}{\Var(X)^{3/2}}\sum_{j\in J_1\cup J_2}
   E\bigl(\bigl|X_j-E(X_j)\bigr|^3\bigr)
 \le \frac{12}{\Var(X)^{3/2}}\sum_{j\in J_1\cup J_2}
   \Var(X_j).
 \ee
\end{proof}

This theorem calls for explicit bounds  for $\Var(X_P)$.  From
\crem{standard}, 
 \begin{align}
 \var(X_P)=&\,\left.\left(z\frac{d}{dz}\right)^2
   \left(p_N\prod_{j=1}^N(z+\eta_j\}\right)\right|_{z=1}\notag\\
  =&\,\sum_{j=1}^N\frac{\eta_j}{(1+\eta_j)^2}
  =\sum_{j=1}^N\frac{\Re(\eta_j)(1+|\eta_j|^2)+2|\eta_j|^2}{|1+\eta_j|^4}
    \label{vareq}
 \end{align}
 Then since $|1+\eta_j|^2=1+2\Re(\eta_j)+|\eta_j|^2\ge1+|\eta_j|^2$
and $|\eta_j|/(1+|\eta_j|^2)\le1/2$,
 \be\label{varub}
 \var(X_P)\le\sum_{j=1}^N\left(\frac{\Re(\eta_j)}{1+|\eta_j|^2}
           +\frac{1}{2}\right)\le N.
 \ee
On the other hand, \eqref{vareq} also yields 
 \be\label{varlb1}
 \var(X_P)\ge \W(X_P):= \frac14\sum_{j=1}^N\frac{\Re(\eta_j)}{1+|\eta_j|^2}.
 \ee
 In our proof of the general case of the LCLT we will need $\var(X_P)$
(respectively $W(X_P)$) to bound $\left|\ex[e^{itX_P}]\right|$ for
``small'' $|t|$ (respectively for ``large'' $|t|$).  Here is a useful upper
bound for $\var(X_P)$.  Introduce $\alpha_P=\max_j|\arg(\eta_j)|$,
$(\alpha_P\in [0,\pi/2])$.  If $\alpha<\pi/2$, then
 \be\label{varub1}
 \var(X_P)\le 4(1+\sec\alpha_P) \W(X_P).
 \ee
Indeed, denoting $r_j=\Re(\eta_j)$, $\alpha_j=|\arg(\eta_j)|$, we bound the $j$-th term in \eqref{vareq} by
\begin{align*}
  \frac{r_j}{1+r_j^2\sec^2\alpha_j}
     &+\frac{2r_j^2\sec^2\alpha_j}{(1+r_j^2\sec^2\alpha_j)^2}\\
  &\hskip30pt\le\, \frac{r_j}{1+r_j^2\sec^2\alpha_j}
     +\frac{2r_j^2\sec^2\alpha_j/(2r_j\sec\alpha_j)}{1+r_j^2\sec^2\alpha_j}\\
   &\hskip30pt\le\,\frac{\Re(\eta_j)}{1+|\eta_j|^2}\cdot(1+\sec\alpha_j),
\end{align*}
and \eqref{varub1} follows. Thus, as $N_P\to\infty$, $\Var(X_P)$ and
$\W(X_{P})$ are of the same order of magnitude if $\alpha_P$ is bounded
away from $\pi/2$.

We will need a lower bound for
$\W(X_P)$ that can make it easier to prove that $\W(X_{P})\to\infty$. To
this end we define, for $P\in\P$,
 \be\label{Deltadef}
 \Delta\;(=\Delta_P):=\min_{1\le j\le N}|\eta_j|\text{Re}(\eta_j),\quad 
    f\;(=f_P):=\frac{p_1}{p_0}.
 \ee
 Notice that $\theta_j:=1/\eta_j$, $j=1,\dots, N$, satisfy
 $$
 \frac{p_N}{\prod_j\theta_j}\cdot \prod_{k=1}^N (z+\theta_k)
   =\sum_{m=0}^Nz^mp_{N-m}.
 $$
 So equating the coefficients by $z^N$ and $z^{N-1}$ we have
 $$
 \frac{p_N}{\prod_j\theta_j}=p_0,
 \quad\frac{p_N}{\prod_j\theta_j}\,\sum_k\theta_k=p_1
     \quad\Longrightarrow\quad\sum_k\theta_k=\frac{p_1}{p_0}.
 $$
 Consequently
 \begin{equation}\label{f=sumtheta}
 f=\frac{p_1}{p_0}=\sum_{j=1}^N\theta_j=\sum_{j=1}^N\text{Re}(\theta_j).
 \end{equation}
In addition,
\begin{equation}\label{Re1/theta}
\text{Re}(\theta_j)=\frac{\text{Re}(\eta_j)}{|\eta_j|^2}
  =|\theta_j|^3\cdot|\eta_j|\text{Re}(\eta_j)
  \ge \Delta |\theta_j|^3.
\end{equation}
Then Jensen's inequality for the convex function $1/(1+x)$, with
\eqref{f=sumtheta} and \eqref{Re1/theta}, yields
 \begin{align}
 \sum_j\frac{\Re(\eta_j)}{1+|\eta_j|^2}
 = &\,\sum_j\frac{\Re(\theta_j)}{1+|\theta_j|^2}
  =f\sum_j\frac{\Re(\theta_j)}{f}\frac1{1+|\theta_j|^2}\notag\\
  \ge&\,\frac{f}{1+f^{-1}\sum_j\Re(\theta_j)|\theta_j|^2}
  \ge\,\frac{f}{1+f^{-1}\sum_j|\theta_j|^3}\notag\\
  \ge&\frac{f}{1+1/\Delta}\ge \frac{f}2\min\{1,\Delta\}.  \label{varlb}
  \end{align}
Thus we have proved

\begin{Lemma}\label{VarX>1,2}
\begin{align}
\Var(X_P)\ge&\, \W(X_P)
   :=\frac{1}{4}\sum_{j=1}^N\frac{\Re(\eta_j)}{1+|\eta_j|^2},\notag\\
  \W(X_P)\ge&\,\frac{f}{8}\min\{1,\Delta\},\notag
 \end{align}
 with $\Delta=\Delta_P$ and $f=f_P$ as defined in \eqref{Deltadef}.
 \end{Lemma}

\subsection{A local central limit theorem: log-concavity case\label{3LCLTs}}

Let us show that the CLT proved in \csect{3CLT} implies an LCLT when the
locations of the roots $\zeta_j$ of the polynomials $P$ (see \eqref{clhp})
are further confined to a sharp wedge enclosing the negative axis in the
complex plane.

\begin{Definition} A sequence $a_n$, $n\ge 0$, of nonnegative real numbers
is {\it log-concave} if
for all $n\ge 1$, $a_n^2\ge a_{n-1}a_{n+1}$.
\end{Definition}

In the factorization \eqref{fact} of $P$ the coefficients $\eta_j$ and $1$
of each linear factor, augmented from the right with an infinite tail of
zeros, obviously form a log-concave sequence, and so do the
coefficients $|\eta_j|^2$, $2\Re(\eta_j)$, and $1$ of each quadratic
factor, provided that
\begin{equation}\label{etajcon}
4(\text{Re}(\eta_j))^2 \ge |\eta_j|^2 \quad\Leftrightarrow\quad
    |\arg(\eta_j)|\le \pi/3.
\end{equation}
In terms of the roots $\zeta_j=-\eta_j$, the last condition is equivalent to 
\begin{equation}\label{zetajcon}
|\arg( \zeta_j)|\in [2\pi/3,\pi],
\end{equation}
for all non-zero roots $\zeta_j$. Since the convolution of log-concave
sequences is log-concave (Menon \cite{Menon}), we see that, under the
condition \eqref{etajcon}, the coefficients of $P$ are also log-concave.
This result appears as a special case in Karlin \cite{Karlin} (Theorem 7.1,
p. 415). (See Stanley \cite{stan} for a more recent, comprehensive, survey
of log-concave sequences.) 

We say that a random variable $X$ taking nonnegative integer values is {\it
log-concave distributed} if the sequence $\{\pr\{X=n\}\}$ is log-concave.
Bender \cite{Bender} discovered that an LCLT holds for a sequence $\{X_n\}$
of log-concave distributed random variables if
$\lim_{n\to\infty}\sup_{x\in \bbr}|F_{X_n}(x)-G(x)|=0$; remarkably, $X_n$
does not have to be a sum of independent random variables.  Later Canfield
\cite{Canfield} quantified Bender's theorem. For this he needed a stronger
notion of log-concavity.

\begin{Definition} A sequence $a_n$, $n\ge 0$, of nonnegative real numbers
is {\it properly\/} log-concave if
 \par\smallskip\noindent
 (i) there exist integers $L$ and $U$ such that $a_n=0$ iff $n< L$ or $n>U$
(in the terminology of \cite{stan}, $\{a_n\}$ has no {\it internal zeros});
 \par\smallskip\noindent
(ii) for all $n\ge 1$, $a_n^2\ge a_{n-1}a_{n+1}$, with equality iff $a_n=0$.
\end{Definition}

Canfield showed that the convolution of properly log-concave sequences is
also properly log-concave.  Observe that the linear and quadratic factors
of our polynomial $P(z)$ are properly log-concave iff
$|\arg( \zeta_j)|\in (2\pi/3,\pi]$. Subject to this stronger condition, the
coefficients of $P(z)$ form therefore a properly log-concave sequence.
 
Here is a slightly simplified formulation of Canfield's result.

\begin{Theorem}(Canfield) Suppose that $X$ has a properly log-concave
distribution and that
$$
\sup_{x\in \bbr}|F_X(x)-G(x)|\le\frac{K}{\sqrt{\Var(X)}}.
$$
If $K>7$, $K/\Var(X)^{1/2}<10^{-7}$, $K/\Var(X)^{1/4}<10^{-2}$, then   
\begin{align*}
\sup_m\left|\pr(X=m)-\frac{1}{\sqrt{2\pi\Var(X)}}\exp\left(-\frac{(m-\ex[X])^2}{2\Var(X)}\right)
\right|
\le \frac{c}{\Var(X)^{3/4}},
\end{align*}
with $c:=14.5K+4.87$.
 \end{Theorem}

This theorem and Theorem \ref{clt2} imply an LCLT for $X_P$ with the roots
$\zeta_j$ satisfying the condition $|\arg( \zeta_j)|\in (2\pi/3,\pi]$.

\begin{Corollary}\label{lcltcon} If the roots $\zeta_j$ of $P(z)$ satisfy
$|\arg( \zeta_j)|\in (2\pi/3,\pi]$, and $\Var(X_P)>144\times10^7$, then
\begin{align*}
\sup_m\left|\pr(X_P=m)-\frac{1}{\sqrt{2\pi\Var(X_P)}}\exp\left(-\frac{(m-\ex[X_p])^2}{2\Var(X_P)}\right)
\right|
\le \frac{180}{\Var(X_P)^{3/4}}.
\end{align*}
\end{Corollary}

\subsection{A local central limit theorem: the general case\label{3LCLT}}

While we proved the LCLT for the roots $\zeta_j$ in the wedge
$|\arg (\zeta_j)|> 2\pi/3$ under a single condition,
$\Var(X_P)\to\infty$, we cannot expect this condition be sufficient in
general. A trivial example is $P(z)$ with purely imaginary, non-zero roots,
in which case the distribution of $X_P$ is supported by the positive even
integers only. We will see shortly, however, that a stronger condition,
$f_P\min\{1,\Delta_P\}\to\infty$ fast enough, does the job perfectly.

\medskip
 
We first state the fundamental estimate, in terms of the variance
$\Var(X_P)$ and its lower bound $\W(X_P)$ defined in \eqref{varlb1}. 

\begin{Theorem}\label{mainthm} Suppose $\Var(X_P)\ge 1$.  Then setting
$X:=X_P$,
\begin{multline}\label{mainest}
\sup_m \left|\pr(X=m) 
   -\frac{1}{\sqrt{2\pi\Var(X)}}
      \exp\left(-\frac{(m-E[X])^2}{2\Var(X)}\right)\right|\\
 \le\frac{\pi}{4^{2/3}}\frac{\Var(X)^{1/3}}{\W(X)}
   \exp\left(-\frac{4^{1/3}}{\pi^2}\frac{\W(X)}{\Var(X)^{2/3}}\right) + 
   \frac{24}{\pi\Var(X)}.
 \end{multline}
\end{Theorem}

\begin{Corollary}\label{V>Var} If
 \be\label{W>Var,2/3}
\W(X_P)\ge \frac{\pi^2}{3\cdot2^{1/3}}\Var(X_P)^{2/3}\log(\Var(X_P)),
 \ee
 then for $X:=X_P$,
 \begin{equation*}
 \sup_m \left|\pr(X=m) -\frac{1}{\sqrt{2\pi\Var(X)}}
   \exp\left(-\frac{(m-E[X])^2}{2\Var(X)}\right)\right|  
 \le \frac{25}{\pi\Var(X)}.
\end{equation*}
\end{Corollary}
 
 \begin{Remark}\label{npcondition} (a)~For
$|\arg( \zeta_j)| \in (2\pi/3,\pi]$,  the estimate \eqref{varub1} leads to 
$W(X_P)\ge (1/12)\var(X_P)$. Therefore the condition \eqref{W>Var,2/3} is
satisfied for $\var(X_P)>2.2 \times 10^8$, the last number being a close
upper bound for the larger root of
 \[
v=\frac{12\,\pi^2}{3\cdot 2^{1/3}}\,v^{2/3}\log v.
\]
 The resulting error estimate, $25/(\pi\var(X_P))$, is noticeably better
than the bound $180/\var(X_P)^{3/4}$ in Corollary \ref{lcltcon}.

\smallskip\noindent
 (b) In general, by \eqref{varub} and Lemma \ref{VarX>1,2},
 $$
 \Var(X_P)\le N_P,\quad \W(X_P)
  \ge \frac{p_1}{8p_0}\min\{1,\Delta_P\},\quad
   (\Delta_P:=\min_j |\eta_j|\text{Re}(\eta_j)).
 $$
 So the condition \eqref{W>Var,2/3}  is certainly met if 
 \begin{equation}\label{p1/p0Delta>}
 \frac{p_1}{p_0}\min\{1,\Delta_P\}\ge \frac{8\pi^2}{3\cdot2^{1/3}}
\, N_P^{2/3}\log N_P.
 \end{equation}
 \end{Remark}
 
For the proof of Theorem \ref{mainthm} we introduce the characteristic
functions $\phi(t)$ of $X$ and $\phi^*(t)$ of $X^*=X-E[X]$:
$\phi(t):=E[e^{itX}]$ and $\phi^*(t):=E[e^{itX^*}]=e^{-itE[X]}\phi(t)$.
The next two lemmas give estimates for these functions.  In \clem{phiallt1}
we use crucially the fact that all roots of $P(z)$ lie in the left hand
plane; this is also used in the proof of \clem{phidiff}, although some
version of this result could be obtained as in \csect{statmech}, using only
the fact that a neighborhood of $z_0=1$ is free from zeros of $P(z)$.

 \begin{Lemma}\label{phiallt1} For all $t\in [-\pi,\pi]$,
 \begin{equation}\label{maxphi1}
   |\phi(t)| \leq \exp\left(-\frac{4t^2}{\pi^2}\W(X_P)\right).
 \end{equation}
 \end{Lemma}

\begin{proof} First of all,
 \be
\phi(t)=\frac{P(e^{it})}{P(1)}=\prod_j\frac{\eta_j+e^{it}}{\eta_j+1}.
 \ee
 So, using $1+u\le e^u$ for $u$ real, $1-\cos t=2\sin^2(t/2)\ge 2t^2/\pi^2$
for $t\in [-\pi,\pi]$, and $|1+\eta_j|^2\le2(1+|\eta_j|^2)$,
 \begin{align}
 |\phi(t)|^2=&\,\prod_j\frac{|\eta_j+e^{it}|^2}{|\eta_j+1|^2}\nonumber\\
  =&\,\prod_j\left(1+ \frac{2\,\Re\eta_j(\cos t-1) 
       +2\,\Im\eta_j\sin t}{|\eta_j+1|^2}\right)\notag\\
  \leq&\,\exp\left(\sum_j\frac{\,\Re(\eta_j)(\cos t-1)}
       {1+|\eta_j|^2}\right)\notag\\
  \leq&\,\exp\left(-\frac{2t^2}{\pi^2}\sum_j
      \frac{\Re(\eta_j)}{1+|\eta_j|^2}\right).\notag
\end{align}
 Invoking the definition of $W(X_P)$ in (3.8) then yields the bound
(3.17) immediately.  \end{proof} 

Unlike Lemma \ref{phiallt1}, the next claim and its proof are more or
less standard; we give the argument to make presentation more
self-contained. 

 \begin{Lemma}\label{phidiff} If $|t|\le 1$ then
 \begin{equation}\label{pdiff}
   \phi^*(t)= \exp\left(-\frac{t^2}{2}\Var(X) +D(t)\right) 
  \quad\text{with}\quad |D(t)|\le 3|t|^3\Var(X).
 \end{equation}
 \end{Lemma}

\begin{proof} We write $X=\sum_{j\in J_1\cup J_2}X_j$ as in \csect{3CLT}. It is easy to check
that $\Var(X_j)\le 1$,  and $\Var(X_j)=1$ iff $\pr(X_j=0)=\pr(X_j=2)=1/2$. Introducing $X_j^*=X_j-\ex[X_j]$, $j\in J_1\cup J_2$, we write
 \be
 \phi^*(t)=\prod_{j\in J_1\cup J_2}\phi_j^*(t),\quad \phi_j^*(t)
  :=E[e^{itX^*_j}];
 \ee
 here, see Feller \cite[Section XVI.5]{F}, \begin{equation*}
\phi_j^*(t)=1-\frac{t^2}{2}\var(X_j)+R_j(t),\quad |R_j(t)|\le
\frac{|t|^3}{6}\ex\bigl[|X_j^*|^3\bigr]\le \frac{|t|^3}{3}\var(X_j),
\end{equation*} as $|X_j^*|\le 2$. Denoting
$u_j:=\tfrac{t^2}{2}\var(X_j)-R_j(t)$, and using $\var(X_j)\le 1$, we see
that, for $|t|\le 1$,
\begin{align*}
|u_j|\le\frac{t^2}{2}\var(X_j)+\frac{|t|^3}{3}\var(X_j)\le \frac{5}{6}t^2\Var(X_j)\le \frac{5}{6}.
\end{align*}
So, using $\log (1-u)=-\sum_{j>0} u^j/j$, we obtain 
$$
\phi_j^*(t)=\exp\bigl[\log(1-u_j)\bigr]=\exp\bigl[-u_j+S_j(t)\bigr],
$$
where
$$
|S_j(t)|\le\sum_{\ell\ge 2}\frac{|u_j|^{\ell}}{\ell}\le \frac{u_j^2}{2(1-|u_j|)}\le 3u_j^2\le \frac{25}{12}\,t^4\Var(X_j).
$$
Therefore
$$
\phi^*_j(t)=\exp\left[-\frac{t^2}{2}\var(X_j) +D_j(t)\right],
$$
where
\begin{align*}
|D_j(t)|=&\,|R_j(t)+S_j(t)|\\
\le&\,\frac{|t|^3}{3}\Var(X_j)+\frac{25t^4}{12}\Var(X_j)\le 3|t|^3\Var(X_j).
\end{align*}
Consequently, for $|t|\le 1$,
\begin{multline}\label{claim}
\phi^*(t)=\prod_j\phi^*_j(t)
=\exp\left(-\frac{t^2}{2}\sum_j\var(X_j)+D(t)\right)\\
=\exp\left(-\frac{t^2}{2}\var(X)
+D(t)\right),
\end{multline}
with $D(t):=\sum_jD_j(t)$, and 
\begin{equation}\label{and}
|D(t)|\le \sum_j|D_j(t)|\le 3|t|^3\Var(X).
\end{equation}
\bi
\end{proof} 

\begin{proofof}{\cthm{mainthm}} For any $T\in[0,\pi]$ we write
 \begin{eqnarray}\label{est}
\left|\pr(X=m)
   -\frac1{\sqrt{2\pi\Var(X)}}\exp\frac{(m-E[X])^2}{2\Var(X)}\right|
   &&\nonumber\\
 &&\hskip-180pt=\left|\frac1{2\pi}\int_{-\pi}^\pi\phi(t)e^{-itm}\,dt
  -\frac1{2\pi}\int_{-\infty}^\infty e^{-t^2\Var(X)/2}e^{-it(m-E[X])}\,dt
   \right|
  \nonumber\\
 &&\hskip-180pt\le\frac{1}{2\pi}\left|
    \int_{T\le|t|\le\pi}\phi(t)e^{-itm}\,dt\right|\nonumber\\
   &&\hskip-130pt+\frac1{2\pi}\left|\int_{|t|\ge T} 
   e^{-t^2\Var(X)/2}e^{-it(m-E[X])}\,dt\right|\nonumber\\
 &&\hskip-130pt+\frac1{2\pi}\int_{|t|\le T}
   \left|\phi^*(t)-e^{-t^2\Var(X)/2}\right|\,dt.
\end{eqnarray}
 Let us denote the three terms in the final expression in \eqref{est} by
$I_1$, $I_2$, and $I_3$, respectively.  Then from \clem{phiallt1} and the
inequality
 \be
\int_{|y|\geq x}e^{-ay^2/2}\,dy\le \frac{2}{ax}e^{-ax^2/2}
 \ee
 we have, for any $T\in (0,\pi]$, 
 \begin{equation}\label{I1,I2<;Tgen}
 \begin{aligned}
 I_1&\le\frac{\pi}{8\W(X)T}\exp\left(-\frac{4T^2}{\pi^2}\W(X)\right);\\
 I_2&\le\frac{1}{\pi\Var(X)T}\exp\left(-\frac{T^2}{2}\Var(X)\right).
 \end{aligned}
\end{equation}

We now turn to $I_3$. Let us pick $T=(4\Var(X))^{-1/3}$; then $T<1$ since
$\Var(X)\ge 1$.  Also, for $|t|\le T$, $|D(t)|$ in Lemma 3.11 is at most
$3/4<1$.  So using that lemma and the inequality
 $$
|e^x-1|\le \frac{|x|}{1-|x|},\quad (|x|<1),$$
we have that for $|t|\le T$,
\begin{align*}
\bigl|\phi^*(t)-e^{-t^2 \Var(X)/2}\bigr|&\le e^{-t^2 \Var(X)/2}\frac{D(t)}{1-|D(t)|}\\
&\le 24\text{Var}(X) |t|^3e^{-t^2 \Var(X)/2}.
\end{align*}
Therefore
 \begin{equation}\label{I3<}
 I_3\le\frac{12\Var(X)}{2\pi} \int_{-\infty}^\infty |t|^3e^{-t^2\var(X)/2}\,dt
  =\frac{24}{\pi\Var(X)}.
 \end{equation}
For this choice of $T$, the bounds \eqref{I1,I2<;Tgen} become
 \begin{align}
 I_1&\le\frac{\pi 4^{1/3}}{8}\frac{\Var(X)^{1/3}}{\W(X)}
     \exp\left(-\frac{4^{1/3}}{\pi^2}
 \frac{\W(X)}{\Var (X)^{2/3}}\right);\label{I1<;Tspec}\\
 I_2&\le\frac{4^{1/3}}{\pi}\Var(X)^{-2/3}
    \exp\left(-\frac{\Var(X)^{1/3}}{2\cdot 4^{2/3}}\right).\nonumber
 \end{align}
 We notice that the top bound exceeds the bottom bound since
$\Var(X)\ge \W(X)$ and $\pi^2>8$.  Adding the bound \eqref{I3<} and the
double bound \eqref{I1<;Tspec}, we get the bound claimed in Theorem
\ref{mainthm}.  \end{proofof}

\section{Graph-counting polynomials}\label{gcpolys}

 Let $G$ be a finite graph with vertex set $V$ and edge set $E$; an edge
$e\in E$ connects distinct vertices $v_1(e)$ and $v_2(e)$, and different
edges may connect the same two vertices.  We identify the subgraphs of $G$
with the subsets $M\subset E$.  For $v\in V$ we let $d_v$ be the degree of
$v$ in $G$ and $d_M(v)$ be the degree of $v$ in the subgraph $M$; to avoid
trivialities we assume that $d_v>0$ for all $v$.

Now suppose that for each $v\in V$ we choose a finite nonempty subset
$C(v)$ of nonnegative integers and define a set $(C)$ of subgraphs of $G$,
associated with the family $(C(v))_{v\in V}$, by
 \be
M\in(C)\qquad\Leftrightarrow\qquad
d_M(v)\in C(v)\hbox{ for all }v\in V.
 \ee
 We assume throughout that $(C)\ne\emptyset$.  Then the {\it graph-counting
polynomial} associated with $(C)$ is
 \be
P_{(C)}(z)=\sum_{M\in(C)}z^{|M|}.
 \ee
 For example, as discussed in \csect{intro}, if $C(v)=\{0,1\}$ for each
$v\in V$ then $(C)$ corresponds to the set of matchings in $G$ or, in the
language of statistical mechanics, to the set of monomer-dimer
configurations on $G$, while if $C(v)=\{0,1,2\}$ for all $v$ then $(C)$ is
the set of unbranched polymer configurations.  If $C(v)=\{0,2\}$ for all $v$
then the subgraphs in $(C)$ are unions of disjoint circuits.

The proofs of the CLT and LCLT given in later sections depend on
information about the locations of the zeros of the polynomials $P_{(C)}$,
and this can be obtained from corresponding information for certain
subsidiary polynomials associated with the vertices.  Given a nonempty
finite set $C$ of nonnegative integers and a positive integer $d$ we define
 \be\label{pcddef}
 p_{C,d}(z)=\sum _{k\in C}\binom{d}{k}z^k;
 \ee
we will often write $p_v=p_{C(v),d_v}$.  The next
two results control respectively the magnitudes and arguments of the
roots of $P_{(C)}$ in terms of corresponding information for the roots of
the $p_v$.

\begin{Theorem}\label{DR1} Suppose that, for each $v\in V$, there is a
constant $r_v>0$ such that $|\zeta|\ge r_v$ for each root $\zeta$ of
$p_v$.  Then every root $\xi$ of $P_{(C)}$ satisfies $|\xi|\ge R$, where
 \be
R=\min_{e\in E}r_{v_1(e)}r_{v_2(e)}.
 \ee
\end{Theorem}

Notice that $p_{C,d}(0)=0$ if and only if $0\notin C$, so that the
hypotheses of \cthm{DR1} imply that $0\in C(v)$ for each $v\in V$.

\begin{proofof}{\cthm{DR1}} The proof uses Grace's Theorem, the notion of
Asano contraction, and the Asano-Ruelle Lemma; these topics are reviewed
in \capp{Asano-Ruelle}.  Let $E_v\subset E$ be the set of edges of $G$
incident on the vertex $v$.  To each polynomial $p_v$ there corresponds a
unique symmetric multi-affine polynomial $q_v$ in the $d_v$ variables
$(z_{v,e})_{e\in E_v}$ such that
$q_v(z,\ldots,z)=p_v(z)$. Since
$p_v(z)\ne0$ for $|z|<r_v$, Grace's Theorem implies that $q_v\ne0$ if
$|z_{v,e}|<r_v$, $\forall e\in E_v$.  Now we define a multi-affine polynomial
 \be\label{Qdef}
 Q^{(0)}\bigl((\z_{v,e})_{v\in V,e\in E_v}\bigr)
   =\prod_{v\in V}q_v\bigl((\z_{v,e})_{e\in E_v}\bigr)
 \ee
 and generate, by repeated Asano contractions
$(z_{v_1(e),e},z_{v_2(e),e})\to z_e$, a sequence of polynomials
$Q^{(0)},Q^{(1)},\ldots,Q^{(|E|)}$, where $Q^{(k)}$ depends on $k$
variables $z_e$ and $(|E|-k)$ pairs of uncontracted variables
$z_{e,v_1(e)},\,z_{e,v_2(e)}$.  From the Asano-Ruelle Lemma and an
inductive argument, $Q^{(k)}((z_e),(z_{v,e}))\ne0$ when the variables
satisfy $|z_e|<r_{v_1(e)}r_{v_2(e)}$, $|z_{e,v}|<r_v$.  In particular,
$Q^{(|E|)}((z_e)_{e\in E})\ne0$ when $|z_e|<R$ for all $e\in E$.  But
$P_{(C)}(z)=Q^{(|E|)}(z,z,\ldots,z)$, completing the proof.  \end{proofof}

\begin{Theorem}\label{DR2} Suppose that for each $v\in V$ there is an angle
$\phi_v\in[0,\pi/2]$ such that each nonzero root $\zeta$ of $p_v$ satisfies
$|\arg(\zeta)|\in[\pi-\phi_v,\pi]$.  Let
 \be\label{defS}
S=\{\,\theta\in[-\pi,\pi]\mid \exists\,(\theta_v)_{v\in V},\;
   |\theta_v|\le\pi/2-\phi_v,\; 
   \theta_{v_1(e)}+\theta_{v_2(e)}=\theta,\;e\in E\,\}.
\ee
  Then every nonzero root $\xi$ of $P_{(C)}$ satisfies
$|\arg(\xi)|\in[\max S,\pi]$.  \end{Theorem}

Our applications of this theorem will always be those of the next
corollary.

\begin{Corollary}\label{uniform} (a) Suppose that there is an angle
$\phi\in[0,\pi/2]$ such that, for each $v\in V$, each nonzero root $\zeta$ of
$p_v$ satisfies $|\arg(\zeta)|\in[\pi-\phi,\pi]$.  Then every
nonzero root $\xi$ of $P_{(C)}$ satisfies $|\arg(\xi)|\in[\pi-2\phi,\pi]$.

 \smallskip\noindent
 (b) Suppose that the graph $G$ is bipartite, so that $V$ may be partitioned
as $V=V_1\cup V_2$ with each $e\in E$ satisfying $v_1(e)\in V_1$,
$v_2(e)\in V_2$. Suppose further that there are angles
$\phi_1,\phi_2\in[0,\pi/2]$ such that, for each $v\in V_i$, each nonzero root
$\zeta$ of $p_v$ satisfies $|\arg(\zeta)|\in[\pi-\phi_i,\pi]$ for
$i=1,2$.  Then every nonzero root $\xi$ of $P_{(C)}$ satisfies
$|\arg(\xi)|\in[\pi-\phi_1-\phi_2,\pi]$.  \end{Corollary}

\begin{proof} For (a) we see that $\pi-2\phi\in S$ by taking
$\theta_v=\pi/2-\phi$ for all $v\in V$; for (b) we have similarly
  $\pi-\phi_1-\phi_2\in S$ from $\theta_v=\pi/2-\phi_i$ if $v\in V_i$.
 \end{proof}

 \begin{proofof}{\cthm{DR2}} It suffices to consider the case $\max S>0$.
We adopt the notations $q_v$ and $Q^{(k)}$ from the proof of \cthm{DR1},
and for $\varepsilon>0$ define also
$p_{v,\varepsilon}(z)=p_v(z+\varepsilon)$ and
$q_{v,\varepsilon}((z_{v,e})_{e\in
E_v})=q_v((z_{v,e}+\varepsilon)_{e\in E_v})$; $q_{v,\varepsilon}$ is the
unique symmetric multi-affine polynomial such that
$q_{v,\varepsilon}(z,\ldots,z)=p_{v,\varepsilon}(z)$.  We also define
 \be\label{Qepsdef}
 Q^{(0)}_\varepsilon\bigl((\z_{v,e})_{v\in V,e\in E_v}\bigr)
   =\prod_{v\in V}q_{v,\varepsilon }\bigl((\z_{v,e})_{e\in E_v}\bigr),
 \ee
 and let
$Q^{(0)}_\varepsilon,Q^{(1)}_\varepsilon,\ldots,Q^{(|E|)}_\varepsilon$ be
obtained by Asano-Ruelle contractions, as in the proof of \cthm{DR1}.
Finally, we define $P_\varepsilon$ by
$P_\varepsilon(z)=Q^{(|E|)}_{\varepsilon}(z,z,\ldots,z)$.

Fix $\theta$ with $|\theta|<\max S$. We claim that if, for each $e\in E$,
$z_{e}$ belongs to the ray $\rho_\theta=\{e^{i\theta}x\mid x>0\}$, then
$Q_{\varepsilon}^{(|E|)}\bigl((z_e)_{e\in E}\bigr)\ne0$.  It follows then
that $P_\varepsilon(z)\ne0$ for $z\in\rho_\theta$, so that $P_\varepsilon$
does not vanish on the open set
 \be
   G:=\{\,z\in\bbc\mid z\ne0, |\arg(z)|<\max S\,\}.
 \ee
 But $\lim_{\varepsilon\to0}P_\epsilon=P_{(C)}$ uniformly on compacts, and
$P_{(C)}$ does not vanish identically since $(C)\ne\emptyset$. So, by an
application on $G$ of the theorem of Hurwitz, $P_{(C)}(z)\ne0$ if
$z\in G$.  This is the desired conclusion.

 We now prove the claim. Clearly $\max S\le \pi$ and
$S =[-\max S, \max S]$.  Consider $\theta\in (-\max S, \max S)$,
 and  let $(\theta_v)_{v\in V}$ be as in the definition of $S$.  If
$\theta=0$ then we may take $\theta_v=0$ for all $v$.  If $\theta\neq 0$ then
necessarily
 \be\label{mintheta}
\min_{j=1,2}|\theta_{v_j(e)}|<\pi/2\quad\text{for every $e\in E$,}
 \ee
   since otherwise $\theta_{v_1(e)}+\theta_{v_2(e)}=\theta\in (-\pi,\pi)$
is inconsistent with $|\theta_{v_j(e)}|\le \pi/2$.  Thus, whatever the
choice of $\theta$, we may assume that \eqref{mintheta} holds.

Now let $\cal H$  and $\overline{\cal H}$ denote respectively  the
open and closed right half planes, and for $\epsilon>0$ define
 \be
K_\epsilon(v) =-(\epsilon+e^{i\theta_v}\overline{\cal H}).
 \ee
 No root $\zeta$ of $p_v(z)$ can belong to $e^{i\theta_v}{\cal H}$; for
$\zeta=0$ this is trivial and for $\zeta\ne0$ follows from
$|\arg(\zeta)|\in[\pi-\phi_v,\pi]$ and $|\theta_v|\le\pi/2-\phi_v$.  Thus
$p_{v,\varepsilon}(z)\ne0$ if
$z+\varepsilon \in e^{i\theta_v}{\cal H}$, that is, if
$z+\varepsilon \notin -e^{i\theta_v}\overline{{\cal H}}$ or
equivalently if $z\notin K_\epsilon(v)$.  Grace's Theorem then implies that
$q_{v,\epsilon}((z_{ve})_{e\in E_v})\ne0$ if $z_{v,e}\notin K_\epsilon(v)$
for all $e\in E_v$.  Repeatedly using the Asano-Ruelle Lemma, as in the
proof of \cthm{DR1}, we then conclude that
$Q^{(|E|)}\bigl((z_e)_{e\in E}\bigr)\ne0$ if
$z_e\notin -K_\epsilon(v_1(e))\times K_\epsilon(v_2(e))$ for all $e\in E$.

Now, the set $-K_\varepsilon(v_1(e))\times K_\varepsilon(v_2(e))$
and the ray $\rho_{\theta_{v_1(e)}+\theta_{v_2(e)}}=\rho_\theta$ do not
intersect.  Otherwise there would
exist $(s_1\ge 0,t_1)$, $(s_2\ge 0,t_2)$ and $x>0$ such that
 $$
-(\varepsilon +e^{i\theta_{v_1(e)}}(s_1+i t_1))
     (\varepsilon +e^{i\theta_{v_2(e)}}(s_2+it_2))=
   xe^{i(\theta_{v_1(e)}+\theta_{v_2(e)})},
 $$
or equivalently 
 \be \label{yy}
 y_1y_2=\rho e^{i\pi}, \quad 
  y_j=e^{-i\theta_{v_j(e)}}\varepsilon+ (s_j+it_j),\quad j=1,2. 
 \ee
 From the second equation in (4.11) we have $|\arg(y_j)|\le\pi/2$, since
$\text{Re}(y_j)\ge 0$, and from (4.9), strict inequality holds for at
least one value of $j$; this is inconsistent with the first equation in
(4.11).  This completes the proof of the claim.  \end{proofof}

\section{Central limit theorems for graph-counting
polynomials}\label{applications}

In this section we consider various infinite families of graphs, each with
an associated assignment $(C(v))_{v\in V}$ of finite sets to vertices; we
let $\G$ denote such a family and $\P=\P(\G)$ denote the class of
associated graph polynomials, which we now denote by $P_G$.  We will
measure the size of a graph $G$ by the size of its edge set $E=E(G)$ and
let $d_{\max}=d_{\max}(G)$ denote the maximum degree of any vertex of $G$;
for convenience we assume that $d_{\max}\ge2$ (the case $d_{\max}=1$ is
trivial to analyze).

 For simplicity we restrict our attention to the two cases implicit in
\ccor{uniform}, and thus assume that either (a)~there is a fixed angle
$\phi\in[0,\pi/2]$ such that for each graph in $G\in\G$ and each
$v\in V(G)$, every nonzero root $\zeta$ of $p_v$ satisfies
$|\arg(\zeta)|\in[\pi-\phi,\pi]$, or (b)~each graph in $\G$ is bipartite,
with $V(G)$ partitioned as $V_1(G)\cup V_2(G)$, and there are fixed angles
$\phi_1,\phi_2\in[0,\pi/2]$ such that for each $G$ and each $v\in V_i(G)$,
$i=1,2$, every nonzero root $\zeta$ of $p_v$ satisfies
$|\arg(\zeta)|\in[\pi-\phi_i,\pi]$.  We will give examples in which the
results of \csect{gcpolys} imply that the roots of each $P\in \P$ lie in
the left half plane, and then apply the results of \csect{lhplane} to
obtain a CLT or LCLT for $\P$.

  Note that the proofs of CLT and LCLT in \csect{lhplane} require two sorts
of hypotheses: on the one hand, the roots of the polynomials must lie in
the left hand plane, or in some more restricted region; on the other, the
variance of the random variable $X_P$, or more precisely the related
quantity $\W(X_P)$, must grow sufficiently fast with $N_P$ (see, for
example, \crem{npcondition}).  When the graphs
in the family under consideration have bounded vertex degree the latter
condition is, in our examples, automatically satisfied.  For the more
general situation with unbounded degrees one must impose conditions on
their growth to obtain the result; we will work this out in detail only for
some of our examples.

\begin{Example}\label{dimers} When $C(v)=C=\{0,1\}$ for each vertex $v$ the
admissible edge configurations are matchings or monomer-dimer
configurations, as discussed in the introduction.  It is well known
\cite{HL} that in this case all roots of $P(z)$ lie on the negative real
axis.  This follows also from \ccor{uniform}(a); one may take $\phi=0$
there, using the fact that for any vertex $v$ the vertex polynomial
$p_v(z)=1+d_vz$ has negative real root $-1/d_v$.  To obtain an LCLT from
\ccor{lcltcon} we need to find the quantities $\Delta$ and $f$ defined in
\eqref{Deltadef}.  \ccor{uniform} implies that the roots $-\eta_j$ of
$P_G$ are negative real numbers satisfying $\eta_j>1/d_{max}^2$, so
that $\Delta=\min_{1\le j\le N}|\eta_j|\text{Re}(\eta_j)\ge 1/d_{\max}^4$.
Further, $p_0=1$ and $p_1=|E|$, since any subgraph with exactly one edge is
admissible, so that $f=p_1/p_0=|E|$. Then from \clem{VarX>1,2},
 \begin{equation}\label{V>E/8d^4}
\Var(X)\ge\frac{|E|}{8\,d_{\max}^{4}},
 \end{equation}
 and an LCLT follows immediately from Corollary \ref{V>Var} and Remark
\ref{npcondition}(a), whenever $d_{max}(G)$ grows more slowly
than $|E(G)|^{1/4}$ in the class of graphs $\G$:

\begin{Theorem}\label{dimerthm} If for each $G\in \G$, $C(v)=\{0,1\}$ for
each vertex $v$, and $|E(G)|\ge 2.2\times 10^8\,d_{\max}^4(G)$, then
\begin{align*}
\sup_m\left|\pr(X_P=m)-\frac{e^{-\tfrac{(m-\ex[X_P])^2}{2\Var(X_P)}}}{\sqrt{2\pi\Var(X_P)}}\right|
\le \frac{200\,d_{\max}^4(G)}{\pi |E(G)|}.
\end{align*}
\end{Theorem}

We note that Godsil \cite{Godsil} used the work of Heilmann and Lieb
\cite{HL} to obtain the estimate $\var(X_P)\ge |E(G)|/(4d_{\rm max}-3)^2$
(see Lemma 3.5 in \cite{Godsil}), and applied Canfield's theorem for
log-concave distributions to get his LCLT for $X_P$ with the error bound
$O\big(d_{\max}^{3/2}/|E|^{3/4}\bigr)$. Godsil's bound is better
(respectively worse) than ours for $d_{\max}\gg |E|^{1/10}$
(respectively for $d_{\max}\ll |E|^{1/10}$).  \end{Example}

\begin{Example}\label{ubsg} When $C(v)=\{0,1,2\}$ for each vertex $v$ the
admissible edge configurations are unbranched subgraphs, as discussed in
the introduction.  In this case the vertex polynomial is
 \be
p_v(z)=1+d_vz+\frac{d_v(d_v-1)}2 z^2.
 \ee
 If $d_v=1$ then $p_v$ has root $\zeta_v=-1$, while if
$d_v\ge2$ the roots are
 \be\label{roots}
  \zeta_v^{\pm}:=\frac{-d_v\pm i\sqrt{d_v^2-2d_v}}{d_v(d_v-1)}.
 \ee
 From $|\zeta_v^{\pm}|^2=2/\bigl(d_v(d_v-1)\bigr)$ we see that each root
 $\zeta$ of $p_v$  satisfies
 \be
 |\zeta|^2\ge \frac2{d_{\max}(d_{\max}-1)};
 \ee
 note that when $d_v=1$ this follows from our convention $d_{\max}\ge2$.
 Thus  from \cthm{DR1} each root $-\eta_j$ of $P_G$ satisfies
 \be
 |\eta_j|\ge \frac2{d_{\max}(d_{\max}-1)}.
 \ee
 Similarly, each root $\zeta$ of $p_v$ satisfies $|\arg(\zeta)|=\pi-\phi_v$
with
 \be
\phi_v  \le\phi_{\max}:=\sin^{-1}\sqrt{\frac{d_{\max}-2}{2(d_{\max}-1)}};
 \ee
 when $d_v=1$ this is trivial and for $d_v\ge2$ follows immediately from
\eqref{roots}.  Thus \ccor{uniform}(a) gives
$|\arg(-\eta_j))|\ge\pi-2\phi_{\max}$.  Since
 \be
\cos(2\phi_{\max})=1-2\sin^2\phi_{max}=\frac1{d_{\max}-1}>0,
 \ee
   all the roots $\-\eta_j$ lie in the left
 half plane; moreover, from \eqref{Deltadef}, 
 \be\label{ubsgdelta}
\Delta=\min_j|\eta_j|\Re(\eta_j)\ge\min_j|\eta_j|^2\cos(2\phi_{\max})
  \ge \frac4{d_{\max}^2(d_{\max}-1)^3}.\ee
 As in \cex{dimers}, $f=p_1/p_0=|E|$, so that from  \clem{VarX>1,2},
 \be\label{ubsgvar}
\Var(X)\ge\frac{|E|}{2\,d_{\max}^2(d_{\max}-1)^3}.
 \ee
  An LCLT then follows from \ccor{V>Var} and \crem{npcondition} when
$d_{max}(G)$ grows logarithmically slower than $|E(G)|^{1/15}$ in the class of
graphs $\G$ (the precise condition is \eqref{fifteen}).

\begin{Theorem}\label{ubsgthm} Suppose that for each $G\in \G$ and vertex
$v$ of $G$, $C(v)$ is $\{0,1\}$ or $\{0,1,2\}$.  If $|E(G)|$ is large
enough so that
 \be\label{fifteen}
  \begin{aligned}
   |E(G)|
    &\ge\frac{2^{2/3}\pi^2}{3}\, d_{\max}^5(G)\la(G)^{2/3}\log \la(G),\\
    \bigl( \la(G)&:=
    \min\{|E(G)|,\,|V(G)|\}\bigr),
  \end{aligned}
     \ee
 (for instance, if $|E(G)|\ge 150\, d_{\text{max}}^{15}(G)\log^3 |V(G)|$), then 
\begin{equation}\label{ubsgcon}
\sup_m\left|\pr(X_P=m)-\frac{e^{-\tfrac{(m-\ex[X_p])^2}{2\Var(X_P)}}}{\sqrt{2\pi\Var(X_P)}}
      \right|
 \le \frac{50\, d_{\max}^{\,5}(G)}{\pi |E(G)|}.
\end{equation}
\end{Theorem}

\begin{proof} By \crem{npcondition}(b), condition \eqref{W>Var,2/3} of
Corollary 3.8 is met if \eqref{p1/p0Delta>} holds, and by $p_1/p_0=|E(G)|$
and (5.8), the latter is true if
\begin{equation}\label{NPlogNP}
\frac{4|E(G)|}{d_{\max}^5(G)}\ge \frac{8\pi^2}{3\cdot2^{1/3}}
   N_{P_G}^{2/3}\log N_{P_G}.
\end{equation}
Now $N_{P_G}\le \la(G)=\min\{|E(G)|,\,|V(G)|\}$, since $2N_{P_G}\le\sum_v c_v\le 2|V(G)|$.
Therefore \eqref{NPlogNP} follows from the condition
(5.10). Thus when (5.10) is satisfied the condition of Corollary 3.8 holds,
and with (5.9) this implies (5.11).  \end{proof}


 %
 %

\begin{Remark}\label{lcubsg} If $d_{\max}(G)\le3$ for all $G\in\G$, for
example if the graphs in $\G$ are all finite subgraphs of the planar
hexagonal lattice, then $\phi_{\max}=\pi/6$ in the above analysis and all
roots $-\eta_j$ of $P_G$ satisfy the condition \eqref{zetajcon} that
$|\arg(-\eta_j)|\in[2\pi/3,\pi]$.  Then from Corollary \ref{V>Var},
Remark \ref{npcondition}(a), and \eqref{V>E/8d^4} we obtain an
LCLT with the error bound $\tfrac{200\cdot 3^4}{\pi |E|}$, provided that
$|E|>1.5\cdot 10^{11}$. \end{Remark} \end{Example}

In the next four examples we consider families of bipartite graphs,
assuming, as discussed above, that the vertex set $V(G)$ of each graph $G$
is partitioned as $V(G)=V_1(G)\cup V_2(G)$.  We assume that there is a
uniform bound on the vertex degrees; specifically, $d_v\le d_i$ for
$v\in V_i(G)$, $i=1,2$, $G\in\G$.  In some cases this assumption is made
for simplicity and one could, in principle, dispense partially or
completely with it, but in others it is strictly necessary, at least for
our methods.

\begin{Example}\label{bipartite1k} Here we take $C_v=\{0,1\}$ for
$v\in V_1(G)$ and, for $v\in V_2(G)$, $C_v=\{0,1,\ldots,k_2\}$ with $k_2$
either 2, 3, or 4.  For $v\in V_1$, $p_v(z)=1+d_vz$ as in \cex{dimers},
with a single negative real root.  Moreover, for $v\in V_2$, each root
$\zeta$ of $p_v(z)$ satisfies $|\arg(\zeta)|\in[\pi-\phi_v,\pi]$, where
$\phi_v\le\phi_{\max}<\pi/2$ for some angle $\phi_{\max}$ which depends on
$k_2$ and $d_2$; for $k_2=2$ this was shown in \cex{ubsg} above (with
$\phi_{\max}=\pi/4$) and for $k_2=3$ or $4$ was shown in \cite{LRS} (see
Theorem~5.1 there).  Thus taking $\phi_1=0$ and $\phi_2=\phi_{\max}$ in
\ccor{uniform}(b) we see that the roots $-\eta_j$ of $P_G$ satisfy
$|\arg(-\eta_j)|\in[\pi-\phi_{\max},\pi]$.  On the other hand, each root
$\zeta$ of any $p_v$ will satisfy $|\zeta|\ge r_0$ for some $r_0>0$, so
that $\Delta=\min_{1\le j\le N}|\eta_j|\text{Re}(\eta_j)\ge\Delta_0>0$
uniformly for all graphs in $\G$; for notational simplicity we may assume
that $\Delta_0\le1$.  We still have $f=p_1/p_0=|E(G)|$, so that
$\Var(X_{P_G})\ge\Delta_0|E|/8$ from \clem{VarX>1,2}. Furthermore, 
by \eqref{varub1},
\[
\Var(X_{P_G})\le (1+\sec\phi_{\max})W(X_{P_G}),
\]
and therefore the condition \eqref{W>Var,2/3}  of Corollary \ref{V>Var} is
satisfied if $\Var(X_{P_G}) \ge v^*$ where $v^*$ is the larger root of
\[
v^{1/3}=\frac{\pi^2(1+\sec\phi_{\max})}{3\cdot 2^{1/3}}\,\ln v.
\]
So for $\Var(X_{P_G}) \ge v^*$ from Corollary \ref{V>Var}
we obtain an LCLT in the form
\begin{align}\label{lcltbipart}
\sup_m\left|\pr(X_P=m)-\frac{e^{-\tfrac{(m-\ex[X_P])^2}{2\Var(X_P)}}}{\sqrt{2\pi\Var(X_P)}}\right|
\le \frac{C}{|E(G)|},
\end{align}
with $C=200/\pi\Delta_0$.

With more precise information on the location of the roots of $p_v$ for
$v\in V_2(G)$ one could extend this result to families in which the vertex
degrees are not bounded, in the style of \cthm{ubsgthm}.  For $k_2=2$ the
necessary information was obtained in the discussion of \cex{ubsg}; for
$k_2=3,4$ one would have to determine the locations of roots of cubic and
quartic polynomials, respectively.  \end{Example}

\begin{Example}\label{bipartite2k} Here $C_v=\{0,1,2\}$ for $v\in V_1(G)$
and $C_v=\{0,1,2,3\}$ for $v\in V_2(G)$, with $d_1$ arbitrary and $d_2\le4$
(the cases $C_2=\{0,\ldots,k_2\}$ with $k_2=1$ or $2$ are covered
by earlier examples).  For $v\in V_1(G)$ a root $\zeta$ of $p_v(z)$ satisfies
$|\arg(\zeta)|<\pi/4$; for $v\in V_2(G)$ all roots of $p_v(z)$ are $\zeta=-1$
when $d_v\le3$, while when $d_v=4$ the roots of $p_v(z)=1+4z+6z^2+4z^3$ are
$-1/2$ and $(-1\pm i)/2$, so that all roots $\zeta$ satisfy
$|\arg(\zeta)|\le\pi/4$.  Thus from \ccor{uniform}(b) the roots $-\eta_j$
of $P_G$ satisfy $|\arg(-\eta_j)|\in[\pi-\phi_{\max},\pi]$ for some
$\phi_{\max}<\pi/2$.  As in \cex{bipartite1k} we find again
$\Delta>\Delta_0$ for some $d_1$-dependent $\Delta_0$, leading to an LCLT of
the form \eqref{lcltbipart}.  Again, one may also find as in \cex{ubsg} an
LCLT for a family of graphs in which $d_1(G)$ can increase with
$|E(G)|$.\end{Example}

\begin{Example}\label{bipartite2kx} This example relies on numerical
computations, although one could probably justify these by obtaining
rigorous bounds.  We take $C_v=\{0,1,2\}$ for for $v\in V_1(G)$ and, for
$v\in V_2(G)$, $C_v=\{0,1,\ldots,k_2\}$ with $k_2$ either $3$ or $4$.  The
possible values of $d_1$ and $d_2$ are shown in
Table~\ref{table:bipartite2kx}; for example, one may take $d_1=3$, $k_2=3$,
and $d_2=5$, $6$, or $7$.  There are a total of five possible examples.
Also shown are angles $\phi_1,\phi_2$, obtained by computation with Maple,
such that for $v\in V_i$ ($i=1,2$), each root $\zeta$ of $p_v(z)$ lies in
$[\pi-\phi_i,\pi]$.  Since in each case $\phi_1+\phi_2<\pi/2$ we obtain an
LCLT of the form \eqref{lcltbipart} as in the two previous examples.

 \begin{table}[ht]
\centering
\renewcommand{\arraystretch}{1.2}\setlength{\doublerulesep}{5pt}
\begin{tabular}{|c|c||c|c||c|c|}
\hline
 \multicolumn2{|c||}\relax&\multicolumn2{c||}{$k_2=3$}&\multicolumn2{c|}{$k_2=4$} \\
\hline
$d_1$&$\phi_1$&$d_2$&$\phi_2$&$d_2$&$\phi_2$\\
\hline\noalign{\vskip-12pt}&&&&&\\ \hline
 3&$0.1666666666\cdots\pi$&5,6,7&$0.3276761158\cdots\pi$&5&$0.30\pi$\\
\hline
 4&$0.1959132762\cdots\pi$&5&$0.2932617986\cdots\pi$&\multicolumn2{c|}\relax\\
\hline
\end{tabular} 

\caption{\label{table:bipartite2kx} Possible values of $d_1$ and $d_2$ with
corresponding values of $\phi_1$ and $\phi_2$.} \end{table}
\end{Example} 

\begin{Example}\label{CLTonly} In the examples considered above, each $C_v$
has been of the form $\{0,1,\ldots,k\}$ for some $k$.  Now we take
$C_v=\{0,1\}$ for $v\in V_1(G)$, but for $v\in V_2(G)$ take $C_v$ to be either
$\{0,2\}$ or $\{0,2,4\}$.  To avoid vertices which are effectively
disconnected from the rest of the graph we assume that $d_v\ge2$ for
$v\in V_2(G)$, and again assume that $d_v\le d_i$ for $v\in V_i(G)$,
$i=1,2$, with $d_1$ and $d_2$ fixed.    Again $p_v(z)$, $v\in V_1$, has a
single negative real root, while  for $v\in V_2(G)$,
$p_v(z)=\tilde p_v(z^2)$, and one finds easily that $\tilde p(w)$, which is
either linear or quadratic, has only negative real roots, so that $p_v$ has
purely imaginary roots. Thus taking $\phi_1=0$ and $\phi_2=\pi/2$ in
\ccor{uniform}(b) we see that the roots $-\eta_j$ of $P_G$ satisfy
$\Re(-\eta_j)\le0$, so that a CLT will follow from \cthm{clt2} once we
verify that $\Var(X_P)\to\infty$ as $N_P\to\infty$ in the family $\P$ under
consideration.

Since in this case the roots $-\eta_j$ of $P$ may lie on the imaginary
axis, the estimates that we have been using for the variance, which begin
with \eqref{varlb1}, are no longer effective.  On the other hand, from
\eqref{vareq} we have
 \be\label{varlbx}
 \var(X_P)\ge \frac12\sum_{j=1}^{N_P}\frac{|\eta_j|^2}{(1+|\eta_j|^2)^2}
  \ge \frac{N_P}2 \min_j \,\bigl\{|\eta_j|^2,|\eta_j|^{-2}\bigr\}.
 \ee
 Since $d_1$ and $d_2$ are fixed we have upper and lower bounds
$0<r\le|\zeta|\le R$ on the magnitudes of the roots $\zeta$ of the
$p_v(z)$, and \cthm{DR1}, together with a corresponding result, with a
similar proof, for upper bounds, implies that $r^2\le|\eta_j|^2\le R^2$.
$N_P$ is the size of the largest admissible configuration of occupied edges
in $G$; let $M\subset E$ be an admissible configuration with $|M|=N_P$.
Each edge of $M$ is incident on a unique vertex of $V_1$, and every vertex
of $V_2$ must be joined by an edge of $E$ to one of these vertices, since
if $v\in V_2$ were not so joined then two edges incident on $v$ could be
added to $M$.  Thus $|V_2|\le d_1N_P$, and since $|E|\le d_2|V_2|$,
$N_P\ge|E|/d_1d_2$.  From \eqref{varlbx} we thus have
 \be
\Var(X_P)\ge \frac{|E|}{d_1d_2}\min\big\{r^2,R^{-2}\big\}.
 \ee
 \end{Example}

 \section{Further central limit theorems} \label{CLTSM}

In this section we give applications of \cthm{stronger}, obtaining central
limit theorems (but not local central limit theorems) in cases in which the
zeros of $P$ avoid a neighborhood of the point $z_0$ on the positive real
axis.  \csect{moregraphs} presents examples for families of graph-counting
polynomials and \csect{lyz} for families of polynomials arising from
statistical mechanics.  To apply the theorem we will establish that, for
the family of polynomials in question, $\Var(x_P)$ grows as $N_P$.  For
this we will use the following result, due to Ginibre \cite{JG}:

\begin{Theorem}\label{ginibre} Let $X$ be a random variable taking 
nonnegative integer values and let
$T_m:=m!\pr\{X=m\}$.
 If for some $A>-1$ and all $m$, $0\le m\le N-2$,
 \be\label{ginhyp}
 \frac{T_{m+2}}{T_{m+1}} \ge \frac{T_{m+1}}{T_m} - A,
 \ee
 then 
 \be\label{ginconc}
 \Var(X) \ge \frac{E[X]}{1 + A} .
 \ee
 \end{Theorem}

\begin{proof} The proof is elementary. Write 
 \be
 E[X]^{2}(1+A)^2 =\left(E\left[\frac{T_{X+1}}{T_{X}} + XA\right]\right)^{2}
  \le E\left[\left(\frac{T_{X+1}}{T_{X}} + XA\right)^{2}\right]
 \ee
  and expand the right hand side, using \eqref{ginhyp} .
\end{proof}

\subsection{Graph-counting polynomials redux}\label{moregraphs}

 In order to apply \cthm{ginibre} to graph-counting polynomials, we show
that \eqref{ginhyp} holds for these under a mild condition on the sets
$C(v)$ defining admissibility of subgraphs.

\begin{Proposition} \label{chkhyp} Suppose that $G$ is a graph with
graph-counting polynomial $P(z)$ and that for each vertex $v$ of $G$,
 \be\label{Cv}
  C(v) = \{0,1,\ldots,k_v\}
 \ee
 for some $k_v\ge1$ Then for all $z_0>0$ the quantities
$T_m=m!p_mz_0^m/P(z_0)$ satisfy \eqref{ginhyp} with $A=(2\alpha+1)z_0$,
where $\alpha:=\max_{v\in V}[d_v-k_v]_+$.  \end{Proposition}

To prove \cprop{chkhyp} we first establish a lemma relating $p_{m+1}$ and
$p_{m+2}$ to $p_m$.  Let $\M_m$ be the set of admissible subgraphs with $m$
edges, so that $p_m=|\M_m|$, and for each $M\in \M_m$ let $K_1(M)$ and
$K_2(M)$ be the number of subgraphs in $\M_{m+1}$ and $\M_{m+2}$,
respectively, which contain $M$; equivalently, we may introduce
 \begin{eqnarray}
 E_1(M)  &=&\{e\mid e\in E\setminus M, \{e\}\cup M\in\M_{m+1}\},\\
 E_2(M)  &=&\{\{e_1,e_2\}\mid e_1,e_2\in E\setminus M,
     \{e_1,e_2\}\cup M\in\M_{m+2}\},
 \end{eqnarray}
 and define $K_1(M)=|E_1(M)|$, $K_2(M)=|E_2(M)|$.  We will 
regard $K_1$ and $K_2$ as random variables, furnishing $\M_m$ with the
uniform probability measure $\Prob(M)=1/p_m$.

\begin{Lemma}\label{count}
 \begin{eqnarray}
 p_{m+1}&=&\frac1{m+1}\sum_{M\subset\M_m}K_1(M)
=\frac{E[K_1]}{m+1}\,p_m,\label{pm1}\\
 p_{m+2}&=&\frac2{(m+2)(m+1)}\sum_{M\subset\M_m}K_2(M)
    =\frac{2E[K_2]}{(m+2)(m+1)}\,p_m.\label{pm2}
 \end{eqnarray}
 \end{Lemma}

\begin{proof} Let $S_1=\{(M,e)\mid M\in\M_m,e\in E_1(M)\}$ and notice that
$|S_1|=\sum_{M\in\M_m}K_1(M)$.  $S_1$ may be put in bijective
correspondence with $S'_1=\{(M',e)\mid M'\in\M_{m+1},e\in M'\}$, via the
correspondence $(M,e)\leftrightarrow(M',e)$ with $M'=M\cup\{e\}$; here we
use the fact that each $C(v)$ has the form \eqref{Cv}, which implies that
the subgraph obtained by deleting an edge from an admissible subgraph is
admissible.  Clearly $|S'_1|=(m+1)p_{m+1}$, and \eqref{pm1} follows from
$|S_1|=|S'_1|$.  Similarly, \eqref{pm2} is obtained from the correspondence
of $S_2=\{(M,\{e_1,e_2\})\mid M\in\M_m,\{e_1,e_2\}\in E_2(M)\}$ with
$S'_2=\{(M',\{e_1,e_2\})\mid M'\in\M_{m+2},e_1,e_2\in M',e_1\ne e_2\}$.
\end{proof}

\begin{proofof}{\cprop{chkhyp}} With $A=(2\alpha+1)z_0$, \eqref{ginhyp}
becomes, from  Lemma~\ref{count},
 \be\label{toshow}
  2E[K_2]-E[K_1]^2 \ge -(2\alpha+1) E[K_1].
 \ee
 Now notice that we may obtain $E_2(M)$ by choosing a pair $\{e_1,e_2\}$ of
edges from $E_1(M)$ and then rejecting this pair if $\{e_1,e_2\}\cup M$ is
not admissible, which can happen only if $e_1$ and $e_2$ share a vertex $v$
with $d_M(v)\ge k_v-1$.  Thus if we first choose $e_1$ with vertices
$v,v'$ we will reject at most $d_v-k_v+d_{v'}-k_{v'}$ ordered edge pairs
$(e_1,e_2)$; this counts unordered edge pairs twice, and we thus find that
 \be
K_2(M)\ge\binom{K_1(M)}2 - \alpha K_1(M).\nonumber\\
 \ee
 Thus 
 \be
   2E[K_2]-E[K_1]^2\ge E[K_1^2]-E[K_1]^2-(2\alpha+1) E[K_1],
 \ee
 verifying \eqref{toshow}.
\end{proofof}

\begin{Example}\label{finalgc} Consider a family $\G$ of graphs such that
for each vertex $v$ of any $G\in\G$, $C_v=\{0,1,\ldots,k_v\}$ with
$1\le k_v\le4$, and assume that the maximum degrees of the graphs are
bounded by some fixed $d_{\max}$.  As discussed in \cex{bipartite1k}, there
is then an angle $\phi_{\max}$ (which may depend on $d_{\max}$), with
$0\le \phi_{\max}<\pi/2$, such that, for any $v$, each root $\zeta$ of
$p_v(z)$ satisfies $|\arg(\zeta)|\in[\pi-\phi_{\max},\pi]$.  Thus taking
$\phi=\phi_{\max}$ in \ccor{uniform}(a) we see that the roots $\zeta_j$ of
$P_G$ satisfy $|\arg(\zeta_j)|\in[\pi-2\phi_{\max},\pi]$, and so for any
$z_0>0$ there will be a neighborhood of $z_0$, which can be chosen uniformly
in $G$, which is free from zeros of $P_G$. 

 A CLT for the family $\P(\G)$ will now follow from \cthm{clt2} once we
show that $\Var(X_P)$ grows faster than $N_P^{2/3}$ in $\P(\G)$, and with
\cprop{chkhyp} this will follow from Ginibre's result, \cthm{ginibre}, if we
can show that $E[X_P]$ grows faster than $N_P^{2/3}$.  But in fact it
follows from \cprop{mean} that $E[X_P]\ge MN_P$, once we verify the
hypotheses of that result.  But since for any $P_G$, $p_0=1$ and
$p_0=|E(G)|\ge N_{P_G}$, condition~(i) of the proposition, that
$p_1\ge c_1p_0N_P$, is satisfied with $c_1=1$.  Moreover, since for $v$ a
vertex of any $G\in\G$ the degree $d_v$ is uniformly bounded by $d_{\max}$,
the possible roots of $p_v(z)$ are uniformly bounded away from zero, and by
\ccor{uniform}(b) so are the roots of $P_G$.  This verifies condition~(ii)
and completes the proof of the CLT for $\P(\G)$.

We remark that, although the methods of \csect{gcpolys} do not show that
the roots of the graph-counting polynomials for the graphs considered here
lie in the left half plane, we do not have an example in which we know that
some of these roots in fact lie in the right half plane.
\end{Example}

\subsection{Lee-Yang zeros for Ising spins}\label{lyz} 

We consider an {\it Ising spin system} in a finite subset $\Lambda$ of the
lattice $\bbz^d$, that is, a collection $\underline\sigma$ of spin
variables $\sigma(x)$, $x \in \Lambda$, taking values $\sigma(x)=\pm1$.
Let  $m(\underline\sigma)$ be the number of sites for
which $\sigma=1$ (the number of ``up spins''). The {\it partition
function} of the system is
 \be\label{gpf}
 P(\beta,z; \Lambda) = \sum_{\underline{\sigma}}
z^{m(\underline{\sigma})}e^{-\beta U(\underline{\sigma})}
  =\sum_{m=0}^{|\Lambda|}p_m(\beta;\Lambda)z^m, 
 \ee
 where
 \be\label{pfdef}
 p_m(\beta;\Lambda)=\sum_{\{\underline\sigma\mid m(\underline\sigma)=m\}}
  e^{-\beta U(\underline{\sigma})} \ee
  Here $U(\underline{\sigma})$ is the {\it interaction energy} for the spin
configuration $\underline\sigma$ and $\beta$ is the inverse temperature.
The parameter $z$ is the {\it the magnetic fugacity}, related to the
(uniform) magnetic field $h$ by $z = e^{2\beta h}$.

In this section we will adopt the spin language above because it is the
traditional one for the discussion of the location of the zeros (in the
variable $z$) of $P$.  Alternatively, however, one may make contact with
the discussion in \csect{intro} by viewing this model as a system of
particles, with site $x\in\Lambda$  occupied by a particle if
$\sigma(x)=1$ and empty if $\sigma(x)=-1$; $m(\underline\sigma)$ is then
the total number of particles in the system. 

For finite $\Lambda$ there can be no zeros of $P(\beta,z;\Lambda)$ for the
physically relevant values of the fugacity---those on the positive real
axis. This means that the thermodynamic pressure,
$\Pi(\beta, z; \Lambda) = |\Lambda|^{-1}\log P(\beta, z; \Lambda)$, is real
analytic for all physically relevant fugacities and there can be no {\it
phase transitions}, that is, no non-analyticity in the pressure as a
function of $z$.

The situation is different in the thermodynamic limit
$\Lambda\nearrow\mathbb{Z}^d$. This limit, with translation invariant
interactions 
 \be\label{energy}
U(\underline\sigma)=-\sum_{x\in\bbz^d}\sum_AJ_{A+x}\prod_{y\in A}\sigma(x+y),
 \ee
 where $\sum_A$ runs over subsets $A\subset\bbz^d$ with $0\in A$ and
$|A|\ge2$, and the $J_A$ are real {\it coupling constants}, which we always
 assume for simplicity satisfy $\sum_A|J_A|<\infty$, is the right
model for a macroscopic system containing, say, $10^{23}$ atoms, when we
are not considering surface effects.  In this limit the thermodynamic
pressure is given by
 \be
 \Pi(\beta, z) = \lim_{\Lambda \nearrow \mathbb{Z}^d}
   \frac{\log P(\beta, z; \Lambda)}{|\Lambda|};
 \ee
 the existence of this limit can be proved for very general $J_A$.  In the
limit, however, the zeros of $P(\beta, z; \Lambda)$ can approach the
positive $z$-axis and thus cause singularities in the pressure
$\Pi(\beta, z)$.  This is a standard mechanism for the occurrence of phase
transitions in statistical mechanical systems \cite{YL,DR2}.

Suppose, on the other hand, that $z_0$ is a point of analyticity of
$\Pi(\beta, z)$, so that some neighborhood $|z-z_0|<\delta$ is free of
zeros for $|\Lambda|$ large.  Let $X:=X_{\beta,z_0;\Lambda}$ be the random
variable defined by \eqref{Xdef} with $p_m=p_m(\beta,\Lambda)$ as in
\eqref{gpf}; $X$ is the total number of up spins (or particles) in the
system in $\Lambda$ at fugacity $z_0$ and inverse temperature $\beta$.  If
we assume for the moment that
 \be\label{varcon}
\lim_{\Lambda\nearrow\bbz^d}\Var(X)/|\Lambda|^{2/3}=\infty,
 \ee
 then \cthm{stronger} shows that the family of these random variables, as
$\Lambda$ increases, satisfies a CLT.  Various cases in which such a
fugacity $z_0$ exists are known.  We briefly describe some of these below.

 In a seminal paper \cite{LY}, Lee and Yang proved that for ferromagnetic
pair interactions,
 \be
U(\underline{\sigma}) = -\sum_{x,y\in\Lambda}
J(x,y)\sigma(x)\sigma(y),\quad J(x,y) \ge 0,
 \ee
 all the zeros of $P(z,\beta;\Lambda)$ lie on the unit circle, $|z|=1$.
Translation invariance is not needed here.  In the translation invariant
situation described above, however, the Lee-Yang result implies that
$\Pi(\beta,z)$ is analytic in $z$ for $|z|\ne1$, so that the number of up
spins satisfies a CLT for $z_0\ne1$.  We remark that Ruelle \cite{R3} gave
a general characterization of polynomials satisfying the Lee-Yang property,
that all roots satisfy $|z| = 1$. He showed in particular that for Ising
systems the {\it only} interactions $U(\underline\sigma)$ for which this
property holds for {\it all} $\beta$ are ferromagnetic pair interactions,
the systems covered by the Lee-Yang theorem.  More recent references about
Lee-Yang zeros can be found in \cite{Sokal}.

In the translation invariant case, which we shall consider from now on,
more is known about the analyticity in $z$, at fixed $\beta$, of
$\Pi(\beta,z)$.  One can show in particular \cite{DR2} that
(i)~$\Pi(\beta,z)$ is analytic on the positive real $z$-axis, if $\beta$ is
sufficiently small (no phase transitions at high temperature), and
(ii)~$P(\beta,z;\Lambda)$ is nonzero, and hence $\Pi(\beta,z;\Lambda)$ is
analytic, in a disc $|z|\le R(\beta;\Lambda)$, with
$R(\beta):=\inf_{\Lambda}R(\beta;\Lambda)>0$, for all $\beta>0$, so that
$\Pi(\beta,z)$ is analytic for $|z|<R(\beta)$.  Each of these results
yields a CLT for the corresponding real fugacities $z_0$.

The behavior of the zeros for other interactions has been investigated
extensively, both analytically and numerically (see \cite{LRS,LR} and
references therein). One can show \cite{LRS}, for certain classes of
interactions $U(\underline\sigma)$, that for some $\delta>0$ each zero of
$P(\beta, z; \Lambda)$ satisfies $\Re \zeta<-\delta$; for these systems,
$X_{\beta,\Lambda}$ satisfies the conditions of \ccor{V>Var} and thus an
LCLT. In other cases one can prove \cite{LRS,LR} that for $\beta$ large the
zeros stay away from the positive $z$-axis and $X_{\beta,\Lambda}$ thus
satisfies a CLT by \cthm{clt2}.  Such CLT have been obtained by other
methods; see for example \cite{DT} and the discussion in \cite{Georgii}.

 In some cases in which the zeros do approach the real $z$-axis at some
$z_0$ in the $\Lambda\nearrow\bbz^d$ limit it is known that the
fluctuations in $X_{\beta,z_0;\Lambda}$ are in fact not Gaussian in the
$\Lambda\nearrow\bbz^d$ limit \cite{McWu,MA}.

We finally want to justify the assumption \eqref{varcon} made above.  From
\cprop{mean} we can conclude that $E[X]\ge M|\Lambda|$ for some $M>0$, once
we verify the hypotheses of that result.  Condition (i), that
$p_1\ge c_1p_0|\Lambda|$, follows from \eqref{pfdef}: the sum defining
$p_0(\beta;\Lambda)$ contains only one term and that defining
$p_1(\beta;\Lambda)$ contains $|\Lambda|$ terms, each nonzero, and the
ratio $e^{-\beta U(\underline{\sigma})}/p_0$, for $m(\underline\sigma)=1$,
is independent of $\underline\sigma$ by translation invariance, at least up
to ``boundary effects,'' and these can be ignored for $|\Lambda|$ large.
Condition (ii) follows from the fact, mentioned above, that no zeros of
$P(\beta,z;\Lambda)$ lie in the disc $|z|<R(\beta)$.  With this, Ginibre's
result \cthm{ginibre} gives $\Var(X_{\beta,\Lambda})\ge M|\Lambda|/(1+A)$.
We need to know, of course, that \eqref{ginhyp} holds for the spin systems
under consideration here.  In fact this is true more generally, as we show
in \capp{showginhyp}.

 \bigskip\noindent
 {\bf Acknowledgments.} The work of J.L.L. was supported in part by NSF
Grant DMR 1104500.  The research of B.P. was supported by the NSF under
Grant No.  DMS 1101237.  We thank S. Goldstein for a helpful discussion
and Dima Iofee for bringing \cite{DS} to our attention.

\appendix
\section{Grace's Theorem and Asano contractions}\label{Asano-Ruelle}

\begin{Theorem}[\bf Grace's theorem]\label{Grace} Let $P(z)$ be a complex
polynomial in one variable of degree at most $n$, and let
$Q(z_1,\ldots,z_n)$ be the unique multi-affine symmetric polynomial in $n$
variables such that $Q(z,\ldots,z)=P(z)$.  If the $n$ roots of $P$ are
contained in a closed circular region $K$ and
$z_1\notin K,\ldots,z_n\notin K$, then $Q(z_1,\ldots,z_n)\ne0$.
\end{Theorem}

Here a {\it closed circular region} is a closed subset $K$ of $\bbc$
bounded by a circle or a straight line.  If $P$ is in fact of degree $k$
with $k<n$ then we say that $n-k$ roots of $P$ lie at $\infty$ and take $K$
noncompact.  For a proof of the result see Polya and Szeg\"o \cite[V,
Exercise 145]{PSz}.

\begin{Lemma}[\bf Asano-Ruelle Lemma \cite{A,R2}]\label{ARlem} Let
$K_1,K_2$ be closed subsets of $\bbc$, with $K_1,K_2\not\ni0$.  If $\Phi$
is separately affine in $z_1$ and $z_2$, and if
 $$	\Phi(z_1,z_2)\equiv A+Bz_1+Cz_2+Dz_1z_2\ne0      $$
whenever $z_1\notin K_1$ and $z_2\notin K_2$, then
  $$\tilde\Phi(z)\equiv A+Dz\ne0 $$
whenever $z\notin-K_1\cdot K_2$. 
\end{Lemma} 

Here we have written $-K_1\cdot K_2=\{-uv\mid u\in K_1,v\in K_2\}$.  The
map $\Phi\mapsto\tilde\Phi$ is called {\it Asano contraction}; we denote it
by $(z_1,z_2)\to z$.

\section{Ginibre's theorem for particle systems}\label{showginhyp}

We consider a set $\Lambda$ of $N$ sites and populate these with a random
configuration of distinguishable particles, at most one particle per
site, in such a way that the probability of having exactly $m$ sites
occupied is given as in \eqref{Xdef} by $p_mz_0^m/P(z_0)$, where
$P(z)=\sum_{m=0}^N p_mz^m$ and
 \be\label{pmdef}
 p_m=\frac1{m!}\sum_{Y_m}e^{-U(Y_m)}.
 \ee
 In \eqref{pmdef} the sum is over ordered $m$-tuples $Y_m=(y_1,\ldots,y_m)$
with $y_i\ne y_j$ for $i\ne j$, and $U(Y_m)=U(y_1,\ldots,y_m)$ is the
potential energy of the system when site $y_i$ is occupied by particle $i$,
$i=1,\ldots,m$, and the remaining $N-m$ sites are empty. The energy $U$ is
invariant under permutation of its arguments.  It will be convenient to
allow sums such as that of \eqref{pmdef} to run over all $Y_m\in\Lambda^m$,
so we define $U(y_1,\ldots,y_m)=+\infty$ whenever $y_i=y_j$ for any $i,j$.
Thus the quantity $T_m$ appearing in \eqref{ginhyp} is
 \be\label{Tdef}
 T_m=\frac{z_0^m}{P(z_0)}\sum_{Y_m\in\Lambda^m}e^{-U(Y_m)}.
 \ee

  Let us define functions $V(Y_m|x_{m+1})$ and $W(Y_m|x_{m+1},x_{m+2})$ by
the requirement that they be $+\infty$ when any two arguments coincide, and
otherwise satisfy
 \begin{eqnarray}U(Y_{m+1})&=&U(Y_m)+V(Y_m|y_{m+1}),\\
  U(Y_{m+2})&=&U(Y_m)+V(Y_m|y_{m+1})+V(Y_m|y_{m+2})\nonumber\\
   &&\hskip55pt+W(Y_m|y_{m+1},y_{m+2}).
 \end{eqnarray}
 Note that 
 \be\label{help}
V(Y_{m+1}|y_{m+2})=V(Y_m|y_{m+2})+W(Y_m|y_{m+1},y_{m+2}).
 \ee
 For any function $F(Y_m)$ we define $F_+=\max\{F,0\}$ and
$F_-=\min\{F,0\}$.  With this notation the two key hypotheses needed for
the result are
 \be\label{Dest}
  D:=\sup_{0\le m\le |\Lambda|-2}\sup_{Y_{m+1}\in \Lambda^{m+1}} 
    \sum_{y_{m+2}\in\Lambda}
   \bigl(1-e^{-\beta W_+(Y_m|y_{m+1},y_{m+2})}\bigr)\,dy<\infty,
 \ee
  and 
 \be\label{lbd}
  -B:=\inf_{0\le m\le |\Lambda|-1}\inf_{Y_{m+1}\in \Lambda^m+1} 
   V(Y_m|y_{m+1})>-\infty.
 \ee
 Note that it follows from \eqref{lbd} that for any $m$ and
$Y_{m+2}\in\Lambda^{m+2}$,
 \be\label{modlb}
 V(Y_m|y_{m+1})+W_-(Y_m|y_{m+1},y_{m+2})\ge-B,
 \ee
 since if $W(Y_m|y_{m+1},y_{m+2})\ge0$ then this comes directly from
\eqref{lbd}, while otherwise, with \eqref{help}, it comes from \eqref{lbd}
with $m$ replaced by $m+1$.  We remark that in the spin language of
\csect{lyz} the condition, for translation invariant systems, that
$\sum_A|J_A|<\infty$ (see \eqref{energy}) implies \eqref{modlb}.

\begin{Remark}\label{pairs} These conditions look somewhat artificial for
the general potentials we are considering here, but more natural in the
case of pair interactions, when
$U(Y_m)=\sum_{1\le i\ne j\le m}\phi(y_i,y_j)$.  Then
 \be
D=\sup_{y\in\Lambda}\sum_{x\in\Lambda}\bigl(1-e^{-\beta \phi(x,y)}\bigr)
    \quad\text{and}\quad
-B=\inf_{x\in\Lambda}\inf_{\Lambda'\subset\Lambda}\sum_{y\in\Lambda',\,y\ne x}\phi(x,y).
 \ee
 \end{Remark}

The next result was stated in \cite{JG} but only for the pair potentials
of \crem{pairs}; the proof was not given but was attributed to a private
communication and a preprint.

 \begin{Theorem}\label{genhyp} Suppose that \eqref{Dest} and \eqref{lbd}
  hold.  Then for $m\le N-2$,
 \be\label{ginhypB}
 T_{m+1}^2-T_mT_{m+2}\le ze^{\beta B}D\,T_mT_{m+1}.
 \ee
\end{Theorem}

\begin{proof}  We make a preliminary calculation:
 \begin{eqnarray}\label{decomp}
   e^{-\beta[V(Y_m|x)+W(Y_m|x,y)]}
  &=&e^{-\beta V(Y_m|x)}
    \bigl[\bigl(e^{-\beta W(Y_m|x,y)}-1\bigr)+1\bigr]\nonumber\\
  &&\hskip-100pt \ge e^{-\beta[V(Y_m|x)}\bigl[e^{-\beta W_-(Y_m|x,y)]}
    \bigl(e^{-\beta W_+(Y_m|x,y)}-1\bigr) +1\bigr]\nonumber\\
  &&\hskip-100pt \label{xxx}\ge e^{\beta B}
    \bigl(e^{-\beta W_+(Y_m|x,y)}-1\bigr) +e^{-\beta V(Y_m|x)},
 \end{eqnarray}
 where we have used \eqref{modlb}. Now with this,
 \begin{eqnarray}
T_{m+1}^2-T_mT_{m+2}&=&\frac{z^{2m+2}}{P(z_0)^2}\sum_{X_m\subset\Lambda}
    \sum_{Y_m\subset\Lambda}\sum_{x,y\in\Lambda}
  e^{-\beta[U(X_m)+U(Y_m)+V(Y_m|y)]}
  \nonumber\\
  &&\hskip5pt
 \times \bigl[e^{-\beta V(X_m|x)}-e^{-\beta[V(Y_m|x)+W(Y_m|x,y)]}\bigr]
    \nonumber\\
&&\hskip-60pt\le\frac{z^{2m+2}}{P(z_0)^2}\sum_{X_m\subset\Lambda}
    \sum_{Y_m\subset\Lambda}\sum_{x,y\in\Lambda}
  e^{-\beta[U(X_m)+U(Y_m)+V(Y_m|y)]}
  \nonumber\\
  &&\hskip-45pt
 \times \bigl[\bigl(e^{-\beta V(X_m|x)}-e^{-\beta V(Y_m|x)}\bigr)
    -e^{\beta B}\bigl(e^{-\beta W_+(Y_m|x,y)}-1\bigr)\bigl].\nonumber\\
 &&\hskip-60pt:=\;R_1+R_2,\label{xxy}
 \end{eqnarray}
  where $R_1$ arises from the term
$\bigl(e^{-\beta V(X_m|x)}-e^{-\beta V(Y_m|x)}\bigr)$ and $R_2$ from the
term $-e^{\beta B}\bigl(e^{-\beta W_+(Y_m|x,y)}-1\bigr)$.  We may average
the formula for $R_1$ given in \eqref{xxy} with the equivalent formula
obtained by interchanging the $X_m$ and $Y_m$ summation variables to
obtain
 \begin{eqnarray}
R_1&=&-\frac{z^{2m+2}}{2P(z_0)^2}\sum_{X_m\subset\Lambda}
    \sum_{Y_m\subset\Lambda}
  e^{-\beta[U(X_m)+U(Y_m)]}\nonumber\\
 &&\hskip20pt\times\left[\sum_{x\in\Lambda}
   \bigl(e^{-\beta V(X_m|x)}-e^{-\beta V(Y_m|x)}\bigr)\right]^2\nonumber\\
&\le&0.\label{R1}
 \end{eqnarray}
For $R_2$ we can use \eqref{Dest} to estimate the sum over $x$ and
thus obtain
 \begin{eqnarray}
R_2&\le&e^{\beta B}D\frac{z^{2m+2}}{P(z_0)^2}\sum_{X_m\subset\Lambda}
    \sum_{Y_m\subset\Lambda}\sum_{y\in\Lambda}
  e^{-\beta[U(X_m)+U(Y_m)+V(Y_m|y)]}\nonumber\\
 &=&ze^{\beta B}DT_mT_{m+1}.\label{R2}
 \end{eqnarray}
Now \eqref{ginhypB} follows from  \eqref{R1} and \eqref{R2}.
\end{proof}

\end{document}